\documentclass[a4paper]{amsart}

\usepackage[alphabetic]{amsrefs}
\usepackage{braket,color,graphicx,tikz,hyperref,newtxtext,newtxmath,cleveref}

\theoremstyle{plain}
\newtheorem{theorem}{Theorem}[section]
\newtheorem{lemma}[theorem]{Lemma}

\theoremstyle{definition}

\theoremstyle{remark}
\newtheorem{remark}[theorem]{Remark}
\numberwithin{equation}{section}

\newcommand{\abs}[1]{\lvert#1\rvert}
\newcommand{\lr}[1]{\Bigl(#1\Bigr)}
\newcommand{\Lr}[1]{\left(#1\right)}

\newcommand{\nm}[2]{\|\,#1\,\|_{#2}}
\newcommand{\jump}[1]{[\![#1]\!]}
\newcommand{\vareps}[1]{#1^\varepsilon}
\newcommand{\leps}[1]{#1_l^\varepsilon}
\newcommand{\heps}[1]{#1_h^\varepsilon}

\newcommand{\R}{\mathbb R}
\newcommand{\na}{\nabla}
\begin{document}
\title[HMMs for 4th-order singular perturbations]{Heterogeneous multiscale methods for fourth-order singular perturbations}
\author[Y. L. Liao]{Yulei Liao}
\address{Department of Mathematics, Faculty of Science, National University of Singapore, 10 Lower Kent Ridge Road, Singapore 119076, Singapore}
\email{ylliao@nus.edu.sg}
\author[P. B. Ming]{Pingbing Ming}
\address{LSEC, Institute of Computational Mathematics and Scientific/Engineering Computing, AMSS, Chinese Academy of Sciences, Beijing 100190, China}
\address{School of Mathematical Sciences, University of Chinese Academy of Sciences, Beijing 100049, China}
\email{mpb@lsec.cc.ac.cn}
\thanks{This work was funded by the National Natural Science Foundation of China through Grant No. 12371438.}
\date{\today}
\begin{abstract}
We develop a numerical homogenization method for fourth-order singular perturbation problems within the framework of heterogeneous multiscale method. These problems arise from heterogeneous strain gradient elasticity and elasticity models for architectured materials. We establish an error estimate for the homogenized solution applicable to general media and derive an explicit convergence for the locally periodic media with the fine-scale $\varepsilon$. For cell problems of size $\delta=\mathbb{N}\varepsilon$, the classical resonance error $\mathcal{O}(\varepsilon/\delta)$ can be eliminated due to the dominance of the higher-order operator. Despite the occurrence of boundary layer effects, discretization errors do not necessarily deteriorate for general boundary conditions. Numerical simulations corroborate these theoretical findings.
\end{abstract}
\maketitle

\section{Introduction}\label{sec:intro}
Consider the singular perturbations of the scalar elliptic homogenization problem~\cite{Bensoussan:2011}:
\begin{equation}\label{eq:ueps}
\left\{\begin{aligned}
    \iota^2\Delta^2u^\varepsilon-\mathrm{div}(\vareps{A}\nabla u^\varepsilon)&=f \qquad
    &&\text{in }\Omega,\\ u^\varepsilon=\partial_{\boldsymbol{n}}u^\varepsilon&=0
    \qquad&&\text{on }\partial\Omega,
\end{aligned}\right.
\end{equation}
where $\Omega\subset\mathbb{R}^d$ is a bounded domain, $\partial_{\boldsymbol{n}}$ is the normal derivative, $0<\varepsilon\ll 1$ is a small parameter that signifies explicitly the length scale of the heterogeneity, $0<\iota\ll 1$ is the strength of the singular perturbations, with
\(
\iota\to0\) when $\varepsilon\to 0$, and the coefficient matrix $\vareps{A}$ belongs to a set $\mathcal{M}(\lambda,\Lambda;\Omega)$  defined by 
\begin{equation}\label{eq:cvcv}
\begin{aligned}
\mathcal{M}(\lambda,\Lambda;\Omega)&{:}=\Bigl\{A\in [L^{\infty}(\Omega)]^{d\times d}\mid
A(\boldsymbol{x})\boldsymbol{\xi}\cdot\boldsymbol{\xi}\ge\lambda\abs{\boldsymbol{\xi}}^2,\\
&\quad\qquad A(\boldsymbol{x})\boldsymbol{\xi}\cdot\boldsymbol{\xi}\ge\dfrac{1}{\Lambda}\abs{A(\boldsymbol{x})\boldsymbol{\xi}}^2\quad\text{for all }\boldsymbol{x}\in\Omega\text{ and }\boldsymbol{\xi}\in\mathbb{R}^d\Bigr\}.
\end{aligned}
\end{equation}
The elements of $\mathcal{M}(\lambda,\Lambda;\Omega)$ are not necessarily symmetric. This boundary value problem represents a possible remedy of the shear bands under severe loading for the heterogeneous materials~\cite{Francfort:1994}. The sequence $\vareps{A}$ satisfying~\cref{eq:cvcv} converges to $\bar{A}$ when $\varepsilon\to0$ in the sense of  $H$-convergence~\cite{Murat:2018}, i.e., the solution $u^\varepsilon$ of~\cref{eq:ueps} satisfies
\[
\left\{\begin{aligned}
    u^\varepsilon&\rightharpoonup \bar{u}\quad && \text{weakly in }H_0^1(\Omega),\\
    \iota\Delta u^\varepsilon&\rightharpoonup 0\quad  && \text{weakly in }L^2(\Omega),\\
    \vareps{A}\nabla u^\varepsilon&\rightharpoonup\bar{A}\nabla\bar{u}\quad && \text{weakly in }[L^2(\Omega)]^d,
\end{aligned}\right.
\]
for some $\bar{u}\in H_0^1(\Omega)$ the weak solution of the homogenization problem
\begin{equation}\label{eq:u}\left\{\begin{aligned}
    -\mathrm{div}(\bar{A}\nabla\bar{u})&=f\qquad&& \text{in }\Omega,\\
    \bar{u}&=0\qquad && \text{on }\partial\Omega,
\end{aligned}\right.\end{equation}
where $\bar{A}\in[L^\infty(\Omega)]^{d\times d}$ is the effective matrix. When $A^\varepsilon$ is periodic, we refer to~\cite{Zhikov:1983,Francfort:1994,Tewary:2021} for qualitative results on homogenization of~\cref{eq:ueps}, and the quantitative homogenization results can be found in \cite{Niu:2019,Pastukhova:2020,Niu:2022}.

Computing the full solution of~\cref{eq:ueps} is computational intensive because one has to resolves the fine-scale $\varepsilon$ with a fourth-order conforming discretization. While a sparse operator compression method for~\cref{eq:ueps} is proposed in~\cite{Zhang:2017}, the stiffness matrices quickly become ill-conditioned even for two-dimensional problem. Instead of finding a fine-scale solution, we seek an approximation of the coarse-scale solution that incurs reduced computational cost without resolving the full fine-scale $\varepsilon$. The coarse-scale solution corresponds to the $H$-limit~\cref{eq:u}.

To this end, we shall develop a coarse-scale simulation method within the framework of the heterogeneous multiscale method (HMM) developed by \textsc{E and Engquist} in~\cite{E:2003} to solve~\cref{eq:u}. A comprehensive review of HMM may be found in~\cite{Engquist:2007,Abdulle:2012}. Recently, there have been many interesting works extending HMM to address the Landau-Lifshitz equation in heterogeneous media~\cite{Runborg:2022a, Runborg:2022b} and the rough-wall laminar viscous flow~\cite{Bjorn:2022}, to name just a few. HMM aims to capture the macroscopic behavior of a system without resolving the microscopic details. HMM consists of two key components: the macroscopic solver and the cell problems for retrieving the missing data of the macroscale solver. We choose the linear Lagrange finite element as the macroscopic solver because $\bar{u}$ solves a second-order elliptic boundary value problem~\cref{eq:u}. The missing data in the macroscopic solver is the effective matrix $\bar{A}$, which is determined by solving certain cell problems that typically take the form of~\cref{eq:ueps} without the source term, subject to certain boundary conditions.

The boundary conditions of the cell problems are crucial for the accuracy and efficiency of the overall method. Careful consideration of these boundary conditions ensures that the macroscopic solver accurately captures the essential features of the original problem. Inspired by~\cite{Hill:1963}, we propose four different boundary conditions for the cell problem, allowing for a unified analysis. Numerical  findings indicate that all proposed conditions yield accurate and efficient results, with only marginal difference among them. It is worth mentioning that the method is general and applicable regardless of the explicit formulation of $\iota$ and $A^\varepsilon$. However, the effective matrix $\bar A$ depends on the explicit relationship between $\iota$ and $\varepsilon$, as demonstrated in~\cite[Theorem 1.3]{Francfort:1994} and~\cite{Niu:2022}. 

We complement the method with a comprehensive analysis. Our analysis follows the framework established by \textsc{E, Ming and Zhang} in~\cite{Ming:2005}. Since the numerical effective matrix $A_H\in\mathcal{M}(\lambda,\Lambda;\Omega)$, we establish the overall accuracy in~\cref{lema:macro} which consists of the discretization error of the macroscopic solver and the error caused by the approximation of the effective matrix that
\[
\nm{\nabla(\Bar{u}-u_H)}{L^2(\Omega)}\le C(\lambda,\Lambda,\Bar{u})\bigl(H+e(\mathrm{HMM})\bigr),
\]
where $e(\mathrm{HMM})$ refers to the error caused by estimating the effective matrix. Under the assumption that $\iota=\mu\varepsilon^\gamma$ with $\gamma,\mu>0$ and $\vareps{A}(\boldsymbol{x})=A(\boldsymbol{x},\boldsymbol{x}/\varepsilon)$ , where $A\in [C^{0,1}(\Omega;L^\infty(\mathbb{R}^d)]^{d\times d}$ and $A(\boldsymbol{x},\cdot)$ is $Y{:}=[-1/2,1/2]^d$-periodic,  we analyze $e(\mathrm{HMM})$ in~\cref{thm:eHMM}. The analysis is suitable for cell problems with general boundary conditions including the essential boundary condition, the natural boundary condition, the free boundary condition and the periodic boundary condition. We state the particular result under certain technical assumptions.
\begin{theorem}[Particular result for~\cref{thm:eHMM}]
If $\iota=\mu\varepsilon^\gamma$ with $\gamma>0$, $A^\varepsilon$ is periodic, and $\delta$ is an integer multiple of $\varepsilon$, then
\[e(\mathrm{HMM})\le C(\mu,A)\begin{cases}
        \varepsilon^{2(1-\gamma)} & 0<\gamma<1,\\
        \varepsilon/\delta+h^2/\varepsilon^2 & \gamma=1,\\    \varepsilon/\delta+\varepsilon^{2(\gamma-1)}+h^2/\varepsilon^2 & \gamma>1,\nm{\nabla_{\boldsymbol{y}}A}{L^\infty(\Omega\times Y)}\text{ is bounded}.
\end{cases}
\]   
\end{theorem}

Compared to the estimates of $e(\mathrm{HMM})$ in~\cite[Theorem 1.2]{Ming:2005} and~\cite[Theorem 3.3]{Du:2010scheme} for the second-order elliptic homogenization problem with locally periodic coefficients, the theorem above includes three key terms on its right-hand side. The term $\mathcal{O}(\varepsilon/\delta)$ represents the resonance error, the term $\mathcal{O}(h^2/\varepsilon^2)$ denotes the discretization error caused by numerically solving the cell problems~\cite{Abdulle:2005}. The term $\mathcal{O}(\varepsilon^{2\abs{1-\gamma}})$ captures the combined effect of homogenization and singular perturbation. Notably, the theorem highlights that both resonance and discretization errors vanish when $0 < \gamma < 1$ and $\vareps A$ is periodic. Specifically, when $\gamma =
1$, no interaction term exists. Finally, we successfully overcome the degeneracy of the discretization error for $\gamma > 1$, which is caused by the emergence of the boundary layer effect~\cite{Semple:1992}. These effects lead to a degeneracy in the discretization error for general singular perturbation problems, scaling to $\mathcal{O}(\sqrt{h/\varepsilon})$ with essential boundary conditions, as detailed in~\cref{sec:apd}. Building upon insights from~\cite{Du:2010scheme}, we leverage the solution structure of the cell problem with locally periodic coefficients to demonstrate a discretization error of $\mathcal{O}(h^2/\varepsilon^2)$.

Our analysis diverges from most previous work by involving different partial differential equations on varying scales, necessitating novel techniques to address the combined effects of singular perturbation and homogenization. To estimate $e(\mathrm{HMM})$, the key methodology in~\cite{Ming:2005, Du:2010} employs the first-order approximation of the cell problem~\cref{eq:veps}, constructed by the corrector $\boldsymbol{\chi}$, and estimates the difference within this approximation; see~\S 3. The formulation of $\boldsymbol{\chi}$ has been derived in~\cite{Francfort:1994}. Specifically, $\boldsymbol{\chi}\equiv 0$ when $0<\gamma<1$, suggesting that the resonance error may be eliminated since the first-order approximation adheres to the boundary conditions of the cell problem. Conversely, when $\gamma>1$, the $H$-limit of~\cref{eq:ueps} aligns with the second-order limit, guiding us to estimate between cell problems of these two types. Our objectives are twofold: firstly, to propose a unified analytical framework for various boundary conditions, and secondly, to mitigate the influence of the boundary layer effect. To this end, we employ the modified corrector proposed by~\cite{Niu:2022}. Different treatment are applied to different scenarios, ultimately producing refined results.

The outline of the paper is as follows. In~\cref{sec:HMM}, we introduce the theory of homogenization and the framework of HMM, and show the well-posedness of the proposed method. In~\cref{sec:eHMM}, we derive the error estimate under certain assumptions on $\iota$ and $A^\varepsilon$. In~\cref{sec:numer}, we employ nonconforming finite elements to solve the cell problems and report numerical results for the problem with two-scale coefficients, which are consistent with the theoretical prediction. The potential of HMM for solving problems without scale separation has been demonstrated in~\cite{Ming:2006numerical}. We conclude the study in~\cref{sec:concl}.

Throughout this paper, the constant $C$ may differ from line to line, while it only depends on the constant $\mu$ and properties for $A$, and it is independent of $\varepsilon,\delta,\gamma$ and meshes size $h$.
\section{Heterogeneous multiscale methods}\label{sec:HMM}
We firstly fix some notations. The space $L^2(\Omega)$ of the square-integrable functions defined on a bounded and convex domain $\Omega$ is equipped with the inner product $(\cdot,\cdot)_\Omega$ and the norm $\nm{\cdot}{L^2(\Omega)}$, while $L_0^2(\Omega)$ is the subspace of $L^2(\Omega)$ with vanishing mean. Let $H^m(\Omega)$ be the standard Sobolev space~\cite{AF:2003} with the norm $\nm{\cdot}{H^m(\Omega)}$, while $H_0^m(\Omega)$ is the closure in $H^m(\Omega)$ of $C_0^\infty(\Omega)$. We may drop $\Omega$  in $\nm{\cdot}{H^m(\Omega)}$ when no confusion may occur. For any function $f$ that is integrable over domain $D$, we denote by $\braket{f}_{D}$ the mean of $f$ over domain $D$.

In this part, we introduce a HMM-FEM to solve~\cref{eq:ueps}, and we make no assumption on the explicit relation between $\iota$ and $\varepsilon$, and the coefficient matrix $A^\varepsilon$.

\subsection{Framework of HMM-FEM}
We employ the linear finite triangular element as the macroscopic solver. Extensions to higher-order finite element macroscopic solvers may be found in in~\cite{Ming:2005,Du:2010scheme, LiMingTang:2012}. Let $X_H$ be the finite element space as
\[
X_H{:}=\Set{Z_H\in H_0^1(\Omega)|Z_H|_K\in\mathbb{P}_1(K)\quad\text{for all }K\in\mathcal{T}_H},
\]
where $\mathcal{T}_H$ is a triangulation of $\Omega$, which consists of simplices $K$ with $h_K$ its diameter and $H{:}=\max_{K\in\mathcal{T}_H}H_K$. We assume that $\mathcal{T}_H$ is shape-regular in the sense of Ciarlet-Raviart~\cite{Ciarlet:2002}: there exists a chunkiness parameter $C$ such that $h_K/\rho_K\le C$, where $\rho_K$ is the diameter of the largest ball inscribed into $K$. We also assume that $\mathcal{T}_H$ satisfies the inverse assumption: there exists $C$ such that $H/H_K\le C$. For any $Z_H\in X_H$, we define $Z_l$ as a linear approximation of $Z_H$ at $\boldsymbol{x}_l\in K$, i.e., $Z_l(\boldsymbol{x})=Z_H(\boldsymbol{x}_l)+(\boldsymbol{x}-\boldsymbol{x}_l)\cdot\na Z_H|_K$. The macroscopic solver aims to find $u_H\in X_H$ such that
\begin{equation}\label{eq:uH}
    a_H(u_H,Z_H)=(f,Z_H)_{\Omega}\quad\text{for all }Z_H\in X_H.
\end{equation}
Here $a_H:X_H\times X_H\to\mathbb{R}$ is defined by
\begin{equation}\label{eq:aH}
    a_H(V_H,Z_H){:}=\sum_{K\in\mathcal{T}_H}\abs{K}\sum_{l=1}^L\omega_l\nabla Z_l\cdot A_H(\boldsymbol{x}_l)\nabla V_l,
\end{equation}
where $\omega_l$ and $\boldsymbol{x}_l$ are the quadrature weights and the quadrature nodes in $K$, respectively. The quadrature scheme is assumed to be exact for linear polynomial. See; e.g.,~\cite{Du:2010scheme}. 

It remains to compute $A_H(\boldsymbol{x}_l)\in\mathbb{R}^{d\times d}$. To this end, we solve
\begin{equation}\label{eq:veps}
    \iota^2\Delta^2\vareps{v}-\mathrm{div}(\vareps{A}\nabla \vareps{v})=0\qquad\text{in }I_\delta,
\end{equation}
where the cell $I_\delta{:}=\boldsymbol{x}_l+\delta Y$ with $\delta$ the cell size. To specify the boundary conditions for~\cref{eq:veps}, we constrain the solution as
\[
\vareps v\in H^2(I_\delta)\quad\text{and\quad} \braket{\nabla\vareps v}_{I_\delta}=\nabla V_l.
\]
In practice, the cell problem has to be solved numerically, and we employ conforming finite element method to discretize the above cell problems. For the sake of simplicity, we only consider the lowest order conforming elements such as the reduced Hsieh-Clough-Tocher element~\cite{Bernadou:1981} and the reduced Powell-Sabin element~\cite{Schumaker:2010}, among many others~\cite[\S 6]{Ciarlet:2002}. Let $\mathcal{T}_h$ be a triangulation of $I_\delta$ by the simplices with maximum mesh size $h$, which is assumed to be shape-regular and satisfies the inverse assumption. Let $X_h\subset H^2(I_\delta)$ be one such finite element space associated with the triangulation $\mathcal{T}_h$, and we assume that there exists an interpolation operator $I_h:H^2(I_\delta)\to X_h$ such that for any $v\in H^k(I_\delta)$ with $k=2,3$,
\begin{equation}\label{eq:inter}
\nm{\nabla^j (I-I_h)v}{L^2(I_\delta)}\le Ch^{k-j}\nm{\nabla^kv}{L^2(I_\delta)}\qquad j=0,1,2.
\end{equation}
The existence of such interpolant may be found in~\cite{Ciarlet:2002}. It is worthwhile to mention that certain nonconforming elements such as Specht triangle~\cite{Specht:1988,LiMingWang:2021} and Nilssen-Tai-Winther element~\cite{Tai:2001}, are also suitable candidates to discretize the cell problem. We refer to~\cref{sec:numer} for more discussions on the nonconforming element discretization.

We shall impose four types of boundary conditions on~\cref{eq:aeps}:
\begin{enumerate}
\item Essential boundary condition: $V_h=X_h\cap H_0^2(I_\delta)$;
\item Natural boundary condition: $V_h=X_h\cap H_0^1(I_\delta)$;
\item Free boundary condition: $V_h=\set{z_h\in X_h\cap L_0^2(I_\delta)|\braket{\nabla z_h}_{I_\delta}=\boldsymbol{0}}$;
\item Periodic boundary condition: $V_h=X_h\cap L_0^2(I_\delta)\cap H_{\mathrm{per}}^2(I_\delta)$.
\end{enumerate}

The variational formulation for~\cref{eq:veps} reads as: Find $\heps{v}-V_l\in V_h$ such that
\begin{equation}\label{eq:aeps}
    \vareps{a}(\heps{v},z_h)=0\quad\text{for all }z_h\in V_h,
\end{equation}
where $\vareps{a}:H^2(I_\delta)\times H^2(I_\delta)\to\mathbb{R}$ is
\[
\vareps{a}(v,z){:}=(\vareps{A}\nabla v,\nabla z)_{I_\delta}+\iota^2(\nabla^2v,\nabla^2z)_{I_\delta}\quad\text{for all\quad}v,z\in H^2(I_\delta).
\]
Then $A_H$ is given by
\begin{equation}\label{eq:AH}
    A_H(\boldsymbol{x}_l)\braket{\nabla \heps{v}}_{I_\delta}{:}=\braket{\vareps{A}\nabla \heps{v}}_{I_\delta}.    
\end{equation}

As to~\cref{eq:veps} with the free boundary condition, the constraint is achieved as in~\cite{Yue:2007} by solving the following variational problem: Find $\vareps v_i\in X_h\cap L_0^2(I_\delta)$ such that
\begin{equation}\label{eq:free}
\vareps a(\vareps v_i,z_h)=\int_{\partial I_\delta}n_iz_h\mathrm{d}\sigma(\boldsymbol{x})\qquad\text{for all}\quad z_h\in X_h\cap L_0^2(I_\delta),
\end{equation}
and~\cref{eq:AH} is equivalent to
\[
    A_H(\boldsymbol{x}_l){:}=\begin{pmatrix}
        \braket{\vareps A\nabla\vareps v_1}_{I_\delta} & \dots &\braket{\vareps A\nabla\vareps v_d}_{I_\delta}
    \end{pmatrix}\begin{pmatrix}
        \braket{\nabla\vareps v_1}_{I_\delta} & \dots & \braket{\nabla\vareps v_d}_{I_\delta}
    \end{pmatrix}^{-1}.
\]

The equivalence between the above two formulations is proved in 
\begin{lemma}\label{lema:free}
The matrix $A_H$ obtained by solving the cell problem~\cref{eq:aeps} subjects to free boundary condition are equivalent with solving~\cref{eq:free}.
\end{lemma}

\begin{proof}
Firstly we need to show that the matrix
\[
\begin{pmatrix}
        \braket{\nabla\vareps v_1}_{I_\delta} & \dots & \braket{\nabla\vareps v_d}_{I_\delta}
\end{pmatrix}\]
is nonsingular, where $v_i^\varepsilon\in X_h\cap L_0^2(I_\delta)$ is the solution of~\cref{eq:free}. it suffices to show that a vector $\boldsymbol{k}\in\mathbb R^d$ such that
\[k_1\braket{\nabla\vareps v_1}_{I_\delta}+k_2\braket{\nabla\vareps v_2}_{I_\delta}+\dots+k_d\braket{\nabla\vareps v_d}_{I_\delta}=\boldsymbol{0},
\]
gives $\boldsymbol{k}=\boldsymbol{0}$.

We set $\vareps v_{\boldsymbol{k}}=k_1\vareps v_1+k_2\vareps v_2\dots +k_d\vareps v_d$, then $\braket{\nabla v_{\boldsymbol{k}}^\varepsilon}_{I_\delta}=\boldsymbol{0}$ and $\vareps{v}_{\boldsymbol{k}}\in V_h$. By the linear property of~\cref{eq:free},
\[
a^\varepsilon(\vareps v_{\boldsymbol{k}},z_h)=\sum_{i=1}^dk_i\int_{\partial I_\delta}n_iz_h\mathrm{d}\sigma(\boldsymbol{x})=\abs{I_\delta}\boldsymbol{k}\cdot\braket{\nabla z_h}_{I_\delta}\qquad\text{for all}\quad z_h\in X_h\cap L_0^2(I_\delta).
\]
Firstly taking $z_h=\vareps{v}_{\boldsymbol{k}}$, we obtain $\vareps v_{\boldsymbol{k}}\equiv 0$. Secondly, taking $z_h=x_i-\braket{x_i}_{I_\delta}$, we obtain $k_i=0$ for $i=1,\dots,d$, and hence $\boldsymbol{k}=\boldsymbol{0}$.

Next we let $\vareps v_{\boldsymbol{k}}=k_1\vareps v_1+k_2\vareps v_2\dots +k_d\vareps v_d+\braket{V_l}_{I_\delta}$ such that
\[
k_1\braket{\nabla\vareps v_1}_{I_\delta}+k_2\braket{\nabla\vareps v_2}_{I_\delta}+\dots+k_d\braket{\nabla\vareps v_d}_{I_\delta}=\nabla V_l.
\]
Then $\boldsymbol{k}=(k_1,\cdots,k_d)$ exists and is unique, and $\vareps v_{\boldsymbol{k}}$ solves~\cref{eq:aeps} and
\begin{align*}
    A_H(\boldsymbol{x}_l)\braket{\nabla\vareps v_{\boldsymbol{k}}}_{I_\delta}&=\begin{pmatrix}
        \braket{\vareps a\nabla\vareps v_1}_{I_\delta} & \dots &\braket{\vareps a\nabla\vareps v_d}_{I_\delta}
    \end{pmatrix}\begin{pmatrix}
        \braket{\nabla\vareps v_1}_{I_\delta} & \dots & \braket{\nabla\vareps v_d}_{I_\delta}
    \end{pmatrix}^{-1}\nabla V_l\\
    &=k_1\braket{\vareps A\nabla\vareps  v_1}_{I_\delta}+k_2\braket{\vareps A\nabla\vareps v_2}_{I_\delta}+\dots+k_d\braket{\vareps A\nabla\vareps v_d}_{I_\delta}\\
    &\quad=\braket{\vareps A\nabla\vareps v_{\boldsymbol k}}_{I_\delta}.
    \end{align*}
    This gives~\cref{eq:AH} because $\boldsymbol{k}$ is a constant vector.
\end{proof}
\subsection{Well-posedness of HMM-FEM}
For any $z\in H^2(I_\delta)$, we define the weighted norm
\[
\nm{z}{\iota}{:}=\nm{\nabla z}{L^2(I_\delta)}+\iota\nm{\nabla^2z}{L^2(I_\delta)}.
\]
The wellposedness of~\cref{eq:aeps} is included in the following lemma.
\begin{lemma}
The cell problem~\cref{eq:aeps} admits a unique solution $\heps{v}$ satisfying
\begin{equation}\label{eq:aepsH2}
    \nm{\heps{v}}{\iota}\le(\sqrt{\Lambda}+\sqrt{\Lambda/\lambda})\nm{\nabla V_l}{L^2(I_\delta)}.
\end{equation}
\end{lemma}

\begin{proof}
For any $z_h\in V_h$, by the Poincar\'e inequality, there exists a constant $C_p$ such that
\[
\nm{z_h}{H^1(I_\delta)}\le C_p \nm{\na z_h}{I_\delta}.
\]
Hence, 
\[
\nm{z_h}{\iota}\ge\min(1/C_p,\iota)\nm{z_h}{H^2(I_\delta)}.
\]
This means that for any fixed $\iota$, the weighted norm $\nm{\cdot}{\iota}$ is indeed a norm over $V_h$. Note that $a^\varepsilon$ is bounded and coercive on $V_h$ with norm $\nm{\cdot}{\iota}$, i.e., for any $z_h\in V_h$,
\[
a_h(z_h,z_h)\ge\dfrac12\min(\lambda,1)\nm{z_h}{\iota}^2.
\]
Hence, the cell problem~\cref{eq:aeps} admits a unique solution by the Lax-Milgram theorem.

Next, we choose $z_h=\heps{v}-V_l\in V_h$ in~\cref{eq:aeps} and obtain
\begin{align*}
\vareps{a}(\heps{v},\heps{v})&=(\vareps{A}\nabla \heps{v},\nabla V_l)\le\nm{\vareps{A}\nabla\heps{v}}{L^2(I_\delta)}\nm{\nabla V_l}{L^2(I_\delta)}\\
&\le\sqrt{\Lambda}(\vareps{A}\nabla\heps{v},\nabla\heps{v})^{1/2}\nm{\nabla V_l}{L^2(I_\delta)}.  
\end{align*}
This immediately implies
\[
(\vareps{A}\nabla\heps{v},\nabla\heps{v})\le\Lambda\nm{\na V_l}{L^2(I_\delta)}^2,
\]
and
\[
\iota^2\nm{\nabla^2\heps v}{L^2(I_\delta)}^2\le\sqrt{\Lambda}(\vareps{A}\nabla\heps{v},\nabla\heps{v})^{1/2}\nm{\nabla V_l}{L^2(I_\delta)}
\le\Lambda\nm{\na V_l}{L^2(I_\delta)}^2.
\]
A combination of the above two inequalities implies~\cref{eq:aepsH2}. 
\end{proof}

In the next lemma, we shall show that~\cref{eq:AH} satisfies the Hill's condition~\cite{Hill:1963}.
\begin{lemma}\label{lema:Hill}
There holds    
    \begin{equation}\label{eq:Hill}
    \braket{\nabla\heps z}_{I_\delta}\cdot A_H(\boldsymbol{x}_l)\braket{\nabla \heps{v}}_{I_\delta}=\braket{\nabla\heps z\cdot \vareps{A}\nabla \heps{v}}_{I_\delta}+\iota^2\braket{\nabla^2\heps z:\nabla^2\heps v}_{I_\delta},
  \end{equation}
where $\heps z$ is the solution of~\cref{eq:aeps} with $V_l$ replaced by any linear function $Z_l$.
\end{lemma}

\begin{proof}
It follows from $\heps z-Z_l\in V_h$ that
    \begin{equation}\label{eq:zepsAvrg}
  \braket{\nabla\heps z}_{I_\delta}=\nabla Z_l,
\end{equation}
and using~\cref{eq:aeps} , we obtain
    \[
    \vareps a(\heps v,\heps z)=\vareps a(\heps v,Z_l)=(\vareps A\nabla\heps v,\nabla Z_l)_{I_\delta}=\abs{I_\delta}\nabla Z_l\cdot\braket{\vareps A\nabla\heps v}_{I_\delta},
\]
which gives~\cref{eq:Hill} with~\cref{eq:AH} and~\cref{eq:zepsAvrg}.
\end{proof}

It follows from the Hill's condition that $A_H\in\mathcal{M}(\lambda,\Lambda;\Omega)$. 
\begin{lemma}\label{lema:cvcv}
There holds $A_H\in\mathcal{M}(\lambda,\Lambda;\Omega)$.
\end{lemma}

\begin{proof}
For any $\boldsymbol{\xi}\in\R^d$, let $\heps{v}\in V_h$ be the solution of~\cref{eq:aeps} with $V_l=\boldsymbol{\xi}\cdot\boldsymbol{x}$.  By Hill's condition~\cref{eq:Hill},
  \[
A_H(\boldsymbol{x}_l)\boldsymbol{\xi}\cdot\boldsymbol{\xi}=\braket{\vareps{A}\na\heps{v}\cdot\na \heps{v}}_{I_\delta}+\iota^2\braket{\abs{\nabla^2\heps v}^2}_{I_\delta}
\ge\braket{\vareps{A}\na\heps{v}\cdot\na\heps{v}}_{I_\delta}.
  \]
  Hence,
  \[
A_H(\boldsymbol{x}_l)\boldsymbol{\xi}\cdot\boldsymbol{\xi}\ge\lambda\braket{\abs{\na\heps{v}}^2}_{I_\delta}
\ge\lambda\abs{\braket{\na\heps{v}}_{I_\delta}}^2=\lambda\abs{\boldsymbol{\xi}}^2.
\]
On the other hand,
\begin{align*}
A_H(\boldsymbol{x}_l)\boldsymbol{\xi}\cdot\boldsymbol{\xi}&\ge\dfrac{1}{\Lambda}\braket{\abs{\vareps{A}\na\heps{v}}^2}_{I_\delta}  \ge\dfrac1{\Lambda}\abs{\braket{\vareps{A}\na\heps{v}}_{I_\delta}}^2=\dfrac1{\Lambda}\abs{A_H(\boldsymbol{x}_l)\braket{\na\heps{v}}_{I_\delta}}^2\\ &=\dfrac1{\Lambda}\abs{A_H(\boldsymbol{x}_l)\boldsymbol{\xi}}^2.
\end{align*}
A combination of the above two inequalities gives $A_H\in\mathcal{M}(\lambda,\Lambda;\Omega)$.
\end{proof}

Lemma~\ref{lema:cvcv} gives the existence and uniqueness of $u_H$ by the Lax-Milgram theorem, which immediately implies the error estimate for the homogenized solution.

\begin{lemma}\label{lema:macro}
There exists $C$ independent of $\varepsilon,\iota,\delta$ and $H$ such that
\begin{equation}\label{eq:uHH1}
    \nm{\nabla(\bar{u}-u_H)}{L^2(\Omega)}\le C\Bigl(H\nm{\Bar{u}}{H^2(\Omega)}+e(\mathrm{HMM})\nm{f}{H^{-1}(\Omega)}\Bigr),
\end{equation}
where
\[
e(\mathrm{HMM}){:}=\max_{K\in\mathcal{T}_H,\boldsymbol{x}_l\in K}\abs{(\bar{A}-A_H)(\boldsymbol{x}_l)},
\]
and $\abs{\cdot}$ denotes the $\ell^2$-norm of matrices and vectors.     
\end{lemma}

\begin{proof}
    This is a direct consequence of~\cite[Theorem 1.1]{Ming:2005} and~\cref{lema:cvcv}.
\end{proof}%

For general coefficients, it has been shown in~\cite[Lemma 2.2]{Du:2010} that if all the quadrature nodes $\boldsymbol{x}_l$ are Lebesgue points of $\Bar{A}$, then
\[
\lim_{\delta\to 0}\lim_{\varepsilon\to 0}e(\mathrm{HMM})=0.
\]
In the next section, we will derive how $e(\mathrm{HMM})$ relies on $\delta$ and $\varepsilon$ for locally periodic coefficients.

\section{Error estimates for $e$(HMM) with locally periodic media}\label{sec:eHMM}
In this part, we estimate $e$(HMM) for the locally periodical coefficients $A^\varepsilon$. We assume that $\iota=\mu\varepsilon^\gamma$ with $\gamma>0$ and a positive constant $\mu$ and $\vareps{A}(\boldsymbol{x})=A(\boldsymbol{x},\boldsymbol{x}/\varepsilon)$ where $A\in [C^{0,1}(\Omega;L^\infty(\mathbb{R}^d))]^{d\times d}$ and $A(\boldsymbol{x},\cdot)$ is $Y$-periodic. 
By~\cite[Theorem 6.3, Theorem 14.5]{Bensoussan:2011} and~\cite[Theorem 1.3]{Francfort:1994}, the effective matrix $\bar{A}$ in~\cref{eq:u} is given by
\[
    \bar{A}(\boldsymbol{x})=\int_{Y}(A+A\nabla^\top_{\boldsymbol{y}}\boldsymbol{\chi})(\boldsymbol{x},\boldsymbol{y})\mathrm{d}\boldsymbol{y},  
\]
where the corrector $\boldsymbol{\chi}=\boldsymbol{0}$ when $0<\gamma<1$. When $\gamma=1$, the corrector $\boldsymbol{\chi}\in[L^\infty(\Omega;H^2(Y))]^d$ satisfies
\[
\left\{\begin{aligned}
    \mu^2\Delta^2_{\boldsymbol{y}}\chi_j-\mathrm{div}_{\boldsymbol{y}}(A\nabla_{\boldsymbol{y}}\chi_j)&=\mathrm{div}_{\boldsymbol{y}}\boldsymbol{a}_j&&\text{in }Y,\\
    \chi_j(\boldsymbol{x},\cdot)&\text{ is }Y\text{-periodic}&&\braket{\chi_j}_{Y}=0,
\end{aligned}\right.
\]
where $\boldsymbol{a}_j$ is the $j$-th column of $A$ with $A=[\boldsymbol{a}_1,\boldsymbol{a}_2\dots,\boldsymbol{a}_d]$. 

When $\gamma>1$, the corrector $\boldsymbol{\chi}\in[L^\infty(\Omega;H^1(Y))]^d$ satisfies
\[
\left\{\begin{aligned}
    -\mathrm{div}_{\boldsymbol{y}}(A\nabla_{\boldsymbol{y}}\chi_j)&=\mathrm{div}_{\boldsymbol{y}}\boldsymbol{a}_j&&\text{in }Y,\\
    \chi_j(\boldsymbol{x},\cdot)&\text{ is }Y\text{-periodic}&&\braket{\chi_j}_{Y}=0.
\end{aligned}\right.
\]

\subsection{Preliminary}
Denote by $\leps{v}$ the solution of~\cref{eq:aeps} with $\vareps{A}$ replaced by $\leps{A}{:}=\vareps A(\boldsymbol{x}_l,\cdot/\varepsilon)$; i.e., find $\leps v-V_l\in V_h$ such that
\begin{equation}\label{eq:vleps}
    \leps a(\leps v,z_h)=0\qquad\text{for all}\quad z_h\in V_h,
\end{equation}
where $\leps a$ is the same with $\vareps a$ provided that $\vareps{A}$ replaced by $\leps{A}$. The first order approximation of $\leps{v}$ is 
\[
\leps{V}(\boldsymbol{x}){:}=V_l(\boldsymbol{x})+\varepsilon\boldsymbol{\chi}^\gamma(\boldsymbol{x}_l,\boldsymbol{x}/\varepsilon)\cdot\nabla V_l,
\]
where $\boldsymbol{\chi}^\gamma{:}=\boldsymbol{\chi}=\boldsymbol{0}$ when $0<\gamma<1$, and $\boldsymbol{\chi}^\gamma\in [L^\infty(\Omega;H^2(Y))]^d$ is the solution of
\begin{equation}\label{eq:chi}
\left\{\begin{aligned}
\mu^2\varepsilon^{2(\gamma-1)} \Delta^2_{\boldsymbol{y}}\chi^\gamma_j-\mathrm{div}_{\boldsymbol{y}}(A\nabla_{\boldsymbol{y}}\chi^\gamma_j)
&=\mathrm{div}_{\boldsymbol{y}}\boldsymbol{a}_j\quad&&\text{in\;}Y,\\
\chi^\gamma_j(\boldsymbol{x},\cdot)
&\text{\;is\;}Y\text{-periodic}\quad&&\braket{\chi^{\gamma}_j}_Y=0,
\end{aligned}\right.
\end{equation}
when $\gamma\ge 1$. This corrector is a modification of that defined in~\cite[\S 3]{Niu:2022}. Moreover, there exists $C$ independent of $\varepsilon$ such that
\begin{equation}\label{eq:chiH3}
\begin{aligned}
\nm{\boldsymbol{\chi}^\gamma}{L^\infty(\Omega;H^1(Y))}&+\varepsilon^{\gamma-1}\nm{\nabla^2_{\boldsymbol{y}}\boldsymbol{\chi}^\gamma}{L^\infty(\Omega;L^2(Y))}\\
&\quad+\varepsilon^{2(\gamma-1)}\nm{\nabla^3_{\boldsymbol{y}}\boldsymbol{\chi}^\gamma}{L^\infty(\Omega;L^2(Y))}\le C.
\end{aligned}
\end{equation} 

If $\leps{A}$ is smoother, then~\cref{eq:chiH3} may be improved and $\boldsymbol{\chi}^\gamma\to\boldsymbol{\chi}$ when $\gamma\to\infty$ as shown below.
\begin{lemma}
If $\gamma>1$ and $\nm{\nabla_{\boldsymbol{y}}A}{L^\infty(\Omega\times Y)}$ is bounded, then
\begin{equation}\label{eq:chiH1}
\nm{\nabla_{\boldsymbol{y}}(\boldsymbol{\chi}^\gamma-\boldsymbol{\chi})}{L^\infty(\Omega;L^2(Y))}\le C\varepsilon^{\gamma-1},
\end{equation}
and
\begin{equation}\label{eq:chiH23}
\nm{\boldsymbol{\chi}^\gamma}{L^\infty(\Omega;H^2(Y))}+\varepsilon^{\gamma-1}\nm{\nabla_{\boldsymbol{y}}^3\boldsymbol{\chi}^\gamma}{L^\infty(\Omega;L^2(Y))}\le C.
\end{equation}
\end{lemma}

\begin{proof}
If $\gamma>1$ and $\nm{\nabla_{\boldsymbol{y}}A}{L^\infty(\Omega\times Y)}$ is bounded, then $\nm{\boldsymbol{\chi}}{L^\infty(\Omega;H^2(Y))}$ is also bounded. We rewrite~\cref{eq:chi} as
\begin{equation}\label{eq:lemaChi30}
\begin{cases}
\mu^2\varepsilon^{2(\gamma-1)}\Delta^2_{\boldsymbol{y}}\chi^\gamma_j=\mathrm{div}_{\boldsymbol{y}}(A\nabla_{\boldsymbol{y}}(\chi^\gamma_j-\chi_j)) &\text{in }Y,\\
    \chi^\gamma_j(\boldsymbol{x},\cdot)\text{ is }Y\text{-periodic}&\braket{\chi_j^\gamma}_{Y}=0.
\end{cases}\end{equation}

Multiplying both sides of~\cref{eq:lemaChi30} by $z=\chi_j^\gamma-\chi_j$, integration by parts, we obtain
\begin{align*}
\nm{\nabla_{\boldsymbol{y}}(\chi_j^\gamma-\chi_j)}{L^\infty(\Omega;L^2(Y))}^2&+\varepsilon^{2(\gamma-1)}\nm{\nabla_{\boldsymbol{y}}^2\chi_j^\gamma}{L^\infty(\Omega;L^2(Y))}^2\\
&\le C\varepsilon^{2(\gamma-1)}\nm{\Delta_{\boldsymbol{y}}\chi_j^\gamma}{L^\infty(\Omega;L^2(Y))}\nm{\Delta_{\boldsymbol{y}}\chi_j}{L^\infty(\Omega;L^2(Y))}.
\end{align*}
This gives~\cref{eq:chiH1} and the $H^2$-estimate in~\cref{eq:chiH23}. 

Finally, by the $H^3$-estimate of~\cref{eq:lemaChi30} and~\cref{eq:cvcv},
\begin{align*}
\nm{\chi_j}{L^\infty(\Omega;H^3(Y))}&\le C\varepsilon^{2(1-\gamma)}\nm{\mathrm{div}_{\boldsymbol{y}}(A\nabla_{\boldsymbol{y}}(\chi^\gamma_j-\chi_j))}{L^\infty(\Omega;H^{-1}(Y))}\\
&\le C\varepsilon^{2(1-\gamma)}\nm{\nabla_{\boldsymbol{y}}(\chi^\gamma_j-\chi_j)}{L^\infty(\Omega;L^2(Y))},
\end{align*}
which together with~\cref{eq:chiH1} gives the $H^3$-estimate in~\cref{eq:chiH23}. 
\end{proof}

Next, we define
\begin{equation}\label{eq:barA}
    \bar{A}^\gamma(\boldsymbol{x}){:}=\int_{Y}(A+A\nabla^\top_{\boldsymbol{y}}\boldsymbol{\chi}^\gamma)(\boldsymbol{x},\boldsymbol{y})\mathrm{d}\boldsymbol{y}. 
\end{equation}

If $0<\gamma\le 1$, then $\bar{A}^\gamma=\bar A$. 

If $\gamma>1$ and $\nm{\nabla_{\boldsymbol{y}}A}{L^\infty(\Omega\times Y)}$ is bounded, invoking~\cite[Lemma 3.2]{Niu:2022}, then we get
\begin{equation}\label{eq:homoerr}
    \nm{\bar A^\gamma-\bar A}{L^\infty(\Omega)}\le C\varepsilon^{2(\gamma-1)}.
\end{equation}

Let $\kappa{:}=\lfloor\delta/\varepsilon\rfloor$. By~\cref{eq:chi} and~\cref{eq:barA}, a direct calculation gives
\begin{equation}\label{eq:VlepsAvrg}
    \braket{\nabla \leps{V}}_{I_{\kappa\varepsilon}}=\nabla V_l\quad \text{and\quad}\braket{\leps{A}\nabla \leps{V}}_{I_{\kappa\varepsilon}}=\bar{A}^\gamma(\boldsymbol{x}_l)\nabla V_l.
\end{equation}

Define a cut-off function $\vareps{\rho}\in C_0^\infty(I_{\delta})$ satisfying $0\le\vareps{\rho}\le 1$ and
\begin{equation}\label{eq:rhoeps}
    \vareps{\rho}=1\text{ in }I_{\delta-2\varepsilon},\quad\vareps{\rho}=0\text{ in }\vareps{I_\delta},\quad \abs{\nabla^i\vareps{\rho}}\le C\varepsilon^{-i}\quad\text{for }i=1,2,3,
\end{equation}
where $\vareps{I_\delta}{:}=I_\delta\backslash I_{\delta-\varepsilon}$.

Using the identity~\cref{eq:VlepsAvrg}, we decompose $e$(HMM) into
\begin{equation}\label{eq:eHMM}\begin{aligned}
    (\bar{A}-A_H)(\boldsymbol{x}_l)\nabla V_l&=(\bar A-\bar A^\gamma)(\boldsymbol{x}_l)\nabla V_l+\Lr{\braket{\leps{A}\nabla \leps{V}}_{I_{\kappa\varepsilon}}-\braket{\leps{A}\nabla \leps{V}}_{I_\delta}}\\ &\qquad+\braket{\leps{A}\nabla\leps{\theta}}_{I_\delta}+\Lr{\braket{\leps{A}\nabla \leps{v}}_{I_\delta}-\braket{\vareps{A}\nabla \heps{v}}_{I_\delta}},
    \end{aligned}\end{equation}
where the corrector $\leps{\theta}{:}=\leps{V}-\leps{v}$.

We shall bound the terms in the right-hand side of~\cref{eq:eHMM} in a series of lemmas. The first lemma compares the average of the flux over cells of different sizes.
\begin{lemma}\label{lema:dom}
\[
    \abs{\braket{\leps{A}\nabla \leps{V}}_{I_{\kappa\varepsilon}}-\braket{\leps{A}\nabla \leps{V}}_{I_\delta}}\le C\dfrac\varepsilon\delta\abs{\nabla V_l}.
\]
\end{lemma}

\begin{proof}
If $0<\gamma<1$, then 
\[
\nm{\nabla\leps{V}}{L^2(I_\delta^\varepsilon)}=\nm{\nabla V_l}{L^2(I_\delta^\varepsilon)}=\abs{I_\delta^\varepsilon}^{1/2}\abs{\na V_l}.
\]
If $\gamma\ge 1$, then we use~\cref{eq:chiH3} and obtain
\[
\nm{\nabla\leps{V}}{L^2(I_\delta^\varepsilon)}=\nm{I+\nabla_{\boldsymbol{y}}\boldsymbol{\chi}^\gamma(\boldsymbol{x}_l,\cdot/
\varepsilon)}{L^2(I_\delta^\varepsilon)}\abs{\nabla V_l}\le C\abs{I_\delta^\varepsilon}^{1/2}\abs{\na V_l}.
\]
A direct calculation gives
\begin{align*}
    \abs{\braket{\leps{A}\nabla \leps{V}}_{I_{\kappa\varepsilon}}-\braket{\leps{A}\nabla \leps{V}}_{I_\delta}}&\le\Lr{1-\dfrac{\abs{I_{\kappa\varepsilon}}}{\abs{I_\delta}}}\Lr{\abs{\braket{\leps{A}\nabla\leps{V}}_{I_{\kappa\varepsilon}}}+\abs{\braket{\leps{A}\nabla\leps{V}}_{I_\delta\backslash I_{\kappa\varepsilon}}}}\\
    &\le\dfrac{\abs{I_\delta^\varepsilon}}{\abs{I_\delta}}\abs{\bar{A}^\gamma(\boldsymbol{x}_l)\nabla V_l}+\Lambda\dfrac{\sqrt{\abs{I_\delta^\varepsilon}}}{\abs{I_\delta}}\nm{\nabla\leps{V}}{L^2(I_\delta^\varepsilon)}\\
    &\le C\dfrac\varepsilon\delta\abs{\nabla V_l}.
\end{align*}
This finishes the proof.
\end{proof}

We shall frequently used the following perturbation estimates.
\begin{lemma}\label{lema:loc}
If $A\in[C^{0,1}(\Omega;L^\infty(\mathbb{R}^d))]^{d\times d}$, then
\begin{equation}\label{eq:vlepsH2}
    \nm{\leps v-\heps v}{\iota}\le C\delta\nm{\nabla V_l}{L^2(I_\delta)},
\end{equation}
and
\begin{equation}\label{eq:loc}
\abs{\braket{\leps{A}\nabla \leps{v}}_{I_\delta}-\braket{\vareps{A}\nabla \heps{v}}_{I_\delta}}\le C\delta\abs{\nabla V_l}.
\end{equation}
\end{lemma}

\begin{proof}
Let $z=\leps v-\heps v\in V_h$ in~\cref{eq:vleps}, we obtain
\[
    \leps{a}(\leps v-\heps v,\leps v-\heps v)=((\vareps A-\leps A)\nabla\heps v,\nabla(\leps v-\heps v))_{I_\delta}.
\]
The estimate~\cref{eq:vlepsH2} follows from the above identity and the fact $\abs{\vareps A-\leps A}\le C\delta$.

Using~\cref{eq:aepsH2} and~\cref{eq:vlepsH2}, we obtain
\begin{align*}
    \abs{\braket{\leps{A}\nabla \leps{v}}_{I_\delta}-\braket{\vareps{A}\nabla \heps{v}}_{I_\delta}}&=\abs{\braket{\leps A\nabla(\leps v-\heps v)}_{I_\delta}+\braket{(\leps A-\vareps A)\nabla\heps v}_{I_\delta}}\\
    &\le\dfrac{C}{\sqrt{\abs{I_\delta}}}(\nm{\leps v-\heps v}{\iota}+\delta\nm{\heps v}{\iota})\\
    &\le C\delta\abs{\nabla V_l}.
\end{align*}
This gives~\cref{eq:loc}.
\end{proof}

To estimate the corrector, we define $\leps{w}$ as the adjoint of $\leps{v}$: Find $\leps{w}-W_l\in V_h$ such that
\begin{equation}\label{eq:wleps}
    \leps{a}(z_h,\leps{w})=0\quad\text{for all }z_h\in V_h,
\end{equation}
and $\leps{W}$ is defined the same with $\leps{V}$ except that at the moment $\boldsymbol{\chi}^\gamma$ is the solution of~\cref{eq:chi} with $A$ replacing by $A^\top$.
\subsection{Estimate of the corrector when $0<\gamma<1$}
In what follows, we estimate the corrector $\braket{\leps{A}\nabla\leps{\theta}}_{I_\delta}$, and we start with a trace inequality.
\begin{lemma}
If $z\in H^1(I_\delta)$ and $\braket{z}_{I_\delta}=0$, then
\begin{equation}\label{eq:trc}
    \nm{z}{L^2(\vareps{I_\delta})}\le C\sqrt{\varepsilon\delta}\nm{\nabla z}{L^2(I_\delta)}.
\end{equation}
\end{lemma}

\begin{proof}
We apply the scaling $\hat{\boldsymbol{x}}{:}=\boldsymbol{x}/\delta$ to $I_\delta$ so that the rescaled cell has diameter $1$. Moreover, denote by $\hat{\varepsilon}{:}=\varepsilon/\delta$ and $\hat{z}(\hat{\boldsymbol{x}}){:}=z(\boldsymbol{x})$, it is clear $\braket{\hat{z}}_{I_1}=0$. Using the trace inequality~\cite[Lemma 1.5]{Oleinik:1992mathematical} and the Poincar\'e inequality, we obtain
\[
    \nm{\hat{z}}{L^2(I_1^{\hat{\varepsilon}})}\le C\sqrt{\hat{\varepsilon}}\nm{\hat{z}}{L^2(I_1)}^{1/2}\nm{\hat{z}}{H^1(I_1)}^{1/2}\le C\sqrt{\hat{\varepsilon}}\nm{\nabla\hat{z}}{L^2(I_1)},
\]
and
\[
    \nm{z}{L^2(\vareps{I_\delta})}\le C\delta^{d/2}\nm{\hat{z}}{L^2(I_1^{\hat{\varepsilon}})}\le C\sqrt{\hat{\varepsilon}}\delta^{d/2}\nm{\nabla\hat{z}}{L^2(I_1)}\le C\sqrt{\hat{\varepsilon}}\delta\nm{\nabla z}{L^2(I_\delta)}.
\]
This gives~\cref{eq:trc}.
\end{proof}

If $0<\gamma<1$, then $\leps V=V_l$ and $\leps\theta=V_l-\leps v\in V_h$.
\begin{lemma}
If $\iota=\mu\varepsilon^\gamma$ with $0<\gamma<1$, then there exists $C$ such that
\begin{equation}\label{eq:thtlepsH20}
   \nm{\leps{\theta}}{\iota}\le C\varepsilon^{1-\gamma}\nm{\nabla V_l}{L^2(I_\delta)}. 
\end{equation}
\end{lemma}

\begin{proof}
Choosing $z=\leps\theta\in V_h$ in~\cref{eq:vleps} and using $\braket{\nabla\leps\theta}_{I_\delta}=0$ and the fact that $\bar A(\boldsymbol{x}_l)\nabla V_l$ is a constant vector, we obtain
\begin{equation}\label{eq:thtlepsH200}\begin{aligned}
    \leps a(\leps\theta,\leps\theta)&=\leps a(V_l,\leps\theta)=\Lr{(\leps A-\bar A(\boldsymbol{x}_l))\nabla V_l,\nabla\leps\theta}_{I_\delta}\\
    &=\Lr{(\leps A-\bar A(\boldsymbol{x}_l))\nabla V_l,\vareps\rho\nabla\leps\theta}_{I_\delta}\\
    &\quad+\Lr{(\leps A-\bar A(\boldsymbol{x}_l))\nabla V_l,(1-\vareps\rho)\nabla\leps\theta}_{I_\delta}.
\end{aligned}
\end{equation}

Let $\mathcal{A}\in[L^\infty(\Omega;H^1(Y))]^{d\times d}$ be the solution of
\[\left\{\begin{aligned}
    -\Delta_{\boldsymbol{y}}\mathcal{A}&=A-\bar A&&\text{in }Y,\\
    \mathcal{A}(\boldsymbol{x},\cdot)&\text{ is }Y\text{-periodic}&&\braket{\mathcal{A}}_Y=0.
\end{aligned}\right.
\]
Since $\braket{A(\boldsymbol{x},\cdot)-\bar A(\boldsymbol{x})}_Y=0$, there exists a unique solution $\mathcal{A}$ such that 
\[
\nm{\mathcal{A}}{L^\infty(\Omega;(H^2(Y))}\le C\nm{A-\bar A}{L^\infty(\Omega;L^2(Y))}\le C.
\]
Define $\leps{\mathcal{A}}{:}=\mathcal{A}(\boldsymbol{x}_l,\cdot/\varepsilon)$ and it satisfies
\[
    -\varepsilon^2\Delta\leps{\mathcal{A}}=-\Delta_{\boldsymbol{y}}\mathcal{A}(\boldsymbol{x}_l,\cdot/\varepsilon)=\leps A-\bar A(\boldsymbol{x}_l).
\]
Integration by part, we write
\[
 -\varepsilon^2(\Delta\leps{\mathcal{A}}\nabla V_l,\vareps\rho\nabla\leps\theta)_{I_\delta}=\varepsilon^2(\nabla(\leps{\mathcal{A}}\nabla V_l),\nabla\leps\theta\otimes\nabla\vareps\rho+\vareps\rho\nabla^2\leps\theta)_{I_\delta}.
\]
Using the trace inequality~\cref{eq:trc} with $z=\nabla\leps\theta$ and $\braket{z}_{I_\delta}=0$, we bound the first term in~\cref{eq:thtlepsH200} as
\begin{align*}
\abs{-\varepsilon^2(\Delta\leps{\mathcal{A}}\nabla V_l,\vareps\rho\nabla\leps\theta)_{I_\delta}} 
    &\le C\abs{\nabla V_l}\Bigl(\nm{\nabla_{\boldsymbol{y}}\mathcal{A}(\boldsymbol{x}_l,\cdot/\varepsilon)}{L^2(\vareps{I_\delta})}\nm{\nabla\leps\theta}{L^2(\vareps{I_\delta})}\\
    &\phantom{C\abs{\nabla V_l}}\qquad
    +\varepsilon\nm{\nabla_{\boldsymbol{y}}\mathcal{A}(\boldsymbol{x}_l,\cdot/\varepsilon)}{L^2(I_\delta)}\nm{\nabla^2\leps\theta}{L^2(I_\delta)}\Bigr)\\
    &\le C\varepsilon\nm{\nabla V_l}{L^2(I_\delta)}\nm{\nabla^2\leps\theta}{L^2(I_\delta)}.
\end{align*}

Proceeding along the same line, we bound  the second term in~\cref{eq:thtlepsH200} as
\begin{align*}
    \abs{-\varepsilon^2(\Delta\leps{\mathcal{A}}\nabla V_l,(1-\vareps\rho)\nabla\leps\theta)_{I_\delta}}&\le C\abs{\nabla V_l}\nm{\Delta_{\boldsymbol{y}}\mathcal{A}}{L^2(\vareps{I_\delta})}\nm{\nabla\leps\theta}{L^2(\vareps{I_\delta})}\\
    &\le C\varepsilon\nm{\nabla V_l}{L^2(I_\delta)}\nm{\nabla^2\leps\theta}{L^2(I_\delta)}.
\end{align*}
Substituting the above two inequalities into~\cref{eq:thtlepsH200}, and using 
\(
\iota\nm{\nabla^2\leps\theta}{L^2(I_\delta)}\le\nm{\leps\theta}{\iota},
\)
we obtain~\cref{eq:thtlepsH20}.
\end{proof}

We are ready to estimate $\braket{\leps{A}\nabla\leps{\theta}}_{I_\delta}$ with a dual argument.
\begin{lemma}\label{lema:rsn0}
If $\iota=\mu\varepsilon^\gamma$ with $0<\gamma<1$, then
\begin{equation}\label{eq:lemaRsn0}
    \abs{\braket{\leps{A}\nabla\leps{\theta}}_{I_\delta}}\le C\varepsilon^{2(1-\gamma)}\abs{\nabla V_l}.
\end{equation}
\end{lemma}

\begin{proof}
For any constant vector $\nabla W_l$, we get
\begin{align*}
    \abs{I_\delta}\nabla W_l\cdot\braket{\leps{A}\nabla\leps\theta}_{I_\delta}&=\leps a(\leps\theta,W_l)=\leps a(\leps\theta,W_l-\leps w)\le\nm{\leps\theta}{\iota}\nm{W_l-\leps w}{\iota}\\
    &\le C\varepsilon^{2(1-\gamma)}\nm{\nabla V_l}{L^2(I_\delta)}\nm{\nabla W_l}{L^2(I_\delta)},
\end{align*}
where we have used~\cref{eq:wleps} for $\leps w$ with $z=\leps\theta$ in the second step, and~\cref{eq:thtlepsH20} for both $\leps v$ and $\leps w$ in the last step. This  gives~\cref{eq:lemaRsn0}.
\end{proof}

\begin{remark}
It is worth mentioning that the estimate~\cref{eq:lemaRsn0} is independent of $\delta$ and $h$ when $\iota=\mu\varepsilon^\gamma$ with $0<\gamma<1$, which stands in striking contrast to the corresponding estimate for the second-order homogenization problem; cf.~\cite{Ming:2005,Du:2010}.
\end{remark}

\subsection{Estimate of the corrector when $\gamma\ge1$}
It is clear that $\leps{V}$ satisfies
\begin{equation}\label{eq:Vleps}
    \leps{a}(\leps{V},z)=0\quad\text{for any }z\in H_0^2(I_\delta).
\end{equation}

The following estimates for $\leps{V}$ hang on the a priori estimate~\cref{eq:chiH3}.

\begin{lemma}
If $\iota=\mu\varepsilon^\gamma$ with $\gamma\ge 1$, then
\begin{equation}\label{eq:VlepsH2}
    \nm{(1-\vareps\rho)(\leps V-V_l)}{\iota}\le C\sqrt{\varepsilon/\delta}\nm{\nabla V_l}{L^2(I_\delta)},
\end{equation}
and
\begin{equation}\label{eq:VlepsH3}
    \nm{\nabla \leps{V}}{L^2(\vareps{I_\delta})}+\iota\nm{\nabla^2\leps{V}}{L^2(\vareps{I_\delta})}+\iota^2\nm{\nabla^3 \leps{V}}{L^2(\vareps{I_\delta})}\le C\sqrt{\varepsilon/\delta}\nm{\nabla V_l}{L^2(I_\delta)}.
\end{equation}
\end{lemma}

\begin{proof}
A direct calculation gives
\[\nabla\lr{(1-\vareps\rho)(\leps V-V_l)}=\nabla_{\boldsymbol{y}}^\top\boldsymbol{\chi}^\gamma(\boldsymbol{x}_l,\cdot/\varepsilon)\nabla V_l(1-\vareps{\rho})-\varepsilon\boldsymbol{\chi}^\gamma(\boldsymbol{x}_l,\cdot/\varepsilon)\cdot\nabla V_l\nabla\vareps{\rho},\]
and
\begin{align*}
   &\quad \varepsilon\nabla^2\lr{(1-\vareps\rho)(\leps V-V_l)}=\nabla^2_{\boldsymbol{y}}(\boldsymbol{\chi}^\gamma\cdot\nabla V_l)(\boldsymbol{x}_l,\cdot/\varepsilon)(1-\vareps{\rho})\\
&\quad-\varepsilon\nabla_{\boldsymbol{y}}^\top\boldsymbol{\chi}^\gamma(\boldsymbol{x}_l,\cdot/\varepsilon)\nabla V_l\otimes\nabla\vareps{\rho}-\varepsilon\nabla\vareps{\rho}\otimes\nabla_{\boldsymbol{y}}^\top\boldsymbol{\chi}^\gamma(\boldsymbol{x}_l,\cdot/\varepsilon)\nabla V_l\\
&\quad+\varepsilon^2\boldsymbol{\chi}(\boldsymbol{x}_l,\cdot/\varepsilon)\cdot\nabla V_l\nabla^2\vareps{\rho},
\end{align*}
which together with~\cref{eq:chiH3} and~\cref{eq:rhoeps} leads to
\begin{align*}
    \nm{(1-\vareps\rho)(\leps V-V_l)}{\iota}&\le C\abs{\nabla V_l}\nm{(\abs{\boldsymbol{\chi}}+\abs{\nabla_{\boldsymbol{y}}\boldsymbol{\chi}}+\varepsilon^{\gamma-1}\abs{\nabla^2_{\boldsymbol{y}}
    \boldsymbol{\chi}})(\boldsymbol{x}_l,\cdot/\varepsilon)}{L^2(\vareps{I_\delta})}\\
    &\le C\sqrt{\varepsilon\delta^{d-1}}\abs{\nabla V_l}\\
    &\le C\sqrt{\varepsilon/\delta}\nm{\nabla V_l}{L^2(I_\delta)}.
\end{align*}
This gives~\cref{eq:VlepsH2}.

The proof for~\cref{eq:VlepsH3} may be proceeded in the same way. We omit the details.
\end{proof}

The next lemma concerns the error caused by the cell discretization.
\begin{lemma}
If $\iota=\mu\varepsilon^\gamma$ with $\gamma=1$, or $\gamma>1$ and $\nm{\nabla_{\boldsymbol{y}}A}{L^\infty(\Omega\times Y)}$ is bounded, then
\begin{equation}\label{eq:VlepsInter}
    \nm{(I-I_h)[\vareps\rho(\leps V-V_l)]}{\iota}\le C\dfrac{h}{\varepsilon}\nm{\nabla V_l}{L^2(I_\delta)}.
\end{equation}
\end{lemma}

\begin{proof}
For $k=2,3$, a direct calculation gives 
\begin{align*}
\varepsilon^{k-1}\abs{\nabla^k[\vareps\rho(\leps V-V_l)]}&=\sum_{j=0}^k\varepsilon^j\abs{\nabla^j\vareps\rho}\abs{\nabla_{\boldsymbol{y}}^{k-j}\boldsymbol{\chi}^\gamma(\boldsymbol{x}_l,\cdot/\varepsilon)}\abs{\nabla V_l}\\
&\le C\abs{\nabla V_l}\sum_{j=0}^k\abs{\nabla_{\boldsymbol{y}}^{k-j}\boldsymbol{\chi}^\gamma(\boldsymbol{x}_l,\cdot/\varepsilon)}.
\end{align*}
Using~\cref{eq:chiH3} when $\gamma=1$ and using~\cref{eq:chiH23} when $\gamma>1$ and $\nm{\nabla_{\boldsymbol{y}}A}{L^\infty(\Omega\times Y)}$ is bounded, we obtain
\begin{align*}
\nm{\nabla[\vareps\rho(\leps V-V_l)]}{\iota}&\le C\abs{\nabla V_l}\Bigl(\varepsilon^{-1}\sum_{j=0}^2\nm{\nabla_{\boldsymbol{y}}^j\boldsymbol{\chi}^\gamma(\boldsymbol{x}_l,\cdot/\varepsilon)}{L^2(I_\delta)}\\
&\phantom{C\abs{\nabla V_l}}\qquad+\varepsilon^{\gamma-2}\sum_{j=0}^3\nm{\nabla_{\boldsymbol{y}}^j\boldsymbol{\chi}^\gamma(\boldsymbol{x}_l,\cdot/\varepsilon)}{L^2(I_\delta)}\Bigr)\\
&\le C\varepsilon^{-1}\nm{\nabla V_l}{L^2(I_\delta)}.
\end{align*}

For any $z\in H^3(I_\delta)$, it follows from the interpolation error estimate~\cref{eq:inter} that
\[
\nm{(I-I_h)z}{\iota}\le Ch\nm{\nabla z}{\iota}.
\]
Choosing $z=\vareps\rho(\leps V-V_l)$ and combining the above two inequalities, we obtain~\cref{eq:VlepsInter}.
\end{proof}

\begin{remark}\label{rmk:smooth}
If $\boldsymbol{\chi}\in [L^\infty(\Omega;H^3(Y))]^d$ holds, then the estimate~\cref{eq:chiH23} may be improved to $\nm{\boldsymbol{\chi}^\gamma}{L^\infty(\Omega;H^3(Y))}\le C$, and the interpolation error~\cref{eq:VlepsInter} changes to $\mathcal{O}(h^2/\varepsilon^2)$ when $\gamma\to\infty$. However, this would require extra smoothness assumption on $A$.
\end{remark}

We are ready to prove the estimate of the corrector.
\begin{lemma}
If $\iota=\mu\varepsilon^\gamma$ with $\gamma=1$, or $\gamma>1$ and $\nm{\nabla_{\boldsymbol{y}}A}{L^\infty(\Omega\times Y)}$ is bounded, then
\begin{equation}\label{eq:thtlepsH21}
   \nm{\leps{\theta}}{\iota}\le C(\sqrt{\varepsilon/\delta}+h/\varepsilon)\nm{\nabla V_l}{L^2(I_\delta)}. 
\end{equation}
\end{lemma}

\begin{proof}
A direct calculation gives
\begin{align*}
    \leps{a}(\leps{\theta},\leps{\theta})&=\leps{a}(\leps{V},\leps{\theta})-\leps{a}(\leps{v},\leps{\theta})\\
    &=\iota^2(\nabla^2\leps{V},\nabla^2\leps{\theta})_{I_\delta}+(\leps{A}\nabla \leps{V}-\bar{A}^\gamma(\boldsymbol{x}_l)\nabla V_l, \nabla\leps{\theta})_{I_\delta}\\
    &\qquad+(\bar{A}^\gamma(\boldsymbol{x}_l)\nabla V_l,\nabla\leps{\theta})_{I_\delta}-\leps{a}(\leps{v},\leps{V})+\leps{a}(\leps{v},\leps{v}).
\end{align*}

Using the Hill's condition~\cref{eq:zepsAvrg} and the definition~\cref{eq:vleps} for $\leps{v}$, we substitute $\braket{\nabla\leps{\theta}}_{I_\delta}=\braket{\nabla(\leps{V}-V_l)}_{I_\delta}$ into the third term, and employ $\leps{a}(\leps{v},\leps{v})=\leps{a}(\leps{v},V_l)$ in the last term. Then the above identity is reshaped into
\begin{equation}\label{eq:thtlepsH210}
\begin{aligned}
    \leps{a}(\leps{\theta},\leps{\theta})&=\iota^2(\nabla^2\leps{V},\nabla^2\leps{\theta})_{I_\delta}+\iota^2(\nabla\Delta\leps V,\nabla\leps{\theta})_{I_\delta}\\
    &\qquad+(\leps{A}\nabla\leps{V}-\bar{A}^\gamma(\boldsymbol{x}_l)\nabla V_l-\iota^2\nabla\Delta\leps V, \nabla\leps{\theta})_{I_\delta}\\
    &\qquad+(\bar{A}^\gamma(\boldsymbol{x}_l)\nabla V_l,(\nabla\leps{V}-V_l))_{I_\delta}-\leps{a}(\leps{V},\leps{V}-V_l)\\
    &\qquad+\leps{a}(\leps{\theta},\leps{V}-V_l).
\end{aligned}\end{equation}

Integration by parts, we obtain
\[
(\nabla\Delta\leps V,\vareps{\rho}\nabla\leps{\theta})_{I_\delta}
=-(\vareps{\rho}\nabla^2\leps V,\nabla^2\leps{\theta})_{I_\delta}-(\nabla\vareps{\rho},\nabla^2\leps V\nabla\leps{\theta})_{I_\delta},
\]
from which we write the first line in~\cref{eq:thtlepsH210} as
\begin{align*}
    \iota^2(\nabla^2& \leps{V},\nabla^2\leps{\theta})_{I_\delta}+\iota^2(\nabla\Delta\leps V,\nabla\leps{\theta})_{I_\delta}\\    &=\iota^2(\nabla^2\leps{V},\nabla^2\leps{\theta})_{I_\delta}+\iota^2(\nabla\Delta\leps V,\vareps{\rho}\nabla\leps{\theta})_{I_\delta}+\iota^2((1-\vareps{\rho})\nabla\Delta\leps V,\nabla\leps{\theta})_{I_\delta}\\
    &=\iota^2((1-\vareps{\rho})\nabla^2\leps V,\nabla^2\leps{\theta})_{I_\delta}-\iota^2(\nabla\vareps{\rho},\nabla^2\leps V\nabla\leps{\theta})_{I_\delta}+\iota^2((1-\vareps{\rho})\nabla\Delta\leps V,\nabla\leps{\theta})_{I_\delta}\\
    &\le C\iota\nm{\leps{\theta}}{\iota}(\nm{\nabla^2\leps V}{L^2(I_\delta^\varepsilon)}+\iota\nm{\nabla^3\leps{V}}{L^2(I_\delta^\varepsilon)})\\
    &\le\dfrac14\nm{\leps{\theta}}{\iota}^2+C\dfrac\varepsilon\delta\nm{\nabla V_l}{L^2(I_\delta)}^2,
\end{align*}
where we have used the Young's inequality and~\cref{eq:VlepsH3} in the last step. 

For $i,j=1,\cdots,d$, define the tensor
\[
    B_{ij}(\boldsymbol{x},\boldsymbol{y})=(A_{ij}+A_{ik}\partial_{\boldsymbol{y}_k}\chi^\gamma_j-\mu^2\varepsilon^{2(\gamma-1)}\partial_{\boldsymbol{y}_{ikk}}\chi^\gamma_j-\bar{A}^\gamma_{ij})(\boldsymbol{x},\boldsymbol{y}).
\]
By the definitions~\cref{eq:chi} and~\cref{eq:barA},
\[
\partial_{\boldsymbol{y}_i}B_{ij}=0,
\quad\braket{B_{ij}}_{Y}=0.
\]

By~\cite[Theorem 3.1.1]{Shen:2018}, there exists an anti-symmetric tensor 
\[
\mathcal{B}\in[L^\infty(\Omega;H^1(Y))]^{d\times d\times d}
\]
such that
\[
    \mathcal{B}_{kij}=-\mathcal{B}_{ikj},\qquad \partial_{\boldsymbol{y}_k}\mathcal{B}_{kij}=B_{ij},
\]
and by~\cref{eq:chiH3},
\begin{align*}
    \nm{\mathcal{B}}{L^\infty(\Omega;H^1(Y))}&\le C\nm{B}{L^\infty(\Omega;L^2(Y))}\\
    &\le C\Lr{1+\nm{\nabla_{\boldsymbol{y}}\boldsymbol{\chi}^\gamma}{L^\infty(\Omega;L^2(Y))}+\varepsilon^{2(\gamma-1)}\nm{\nabla_{\boldsymbol{y}}^3\boldsymbol{\chi}^\gamma}{L^\infty(\Omega;L^2(Y))}}\\
    &\le C.
\end{align*}
Define $\leps{B}:=B(\boldsymbol{x}_l,\cdot/\varepsilon)$, and $\leps{\mathcal{B}}:=\mathcal{B}(\boldsymbol{x}_l,\cdot/\varepsilon)$, we obtain
\[
    \leps{A}\nabla\leps{V}-\bar{A}^\gamma(\boldsymbol{x}_l)\nabla V_l-\iota^2\nabla\Delta\leps V=\leps{B}\nabla V_l,
\]
and using the anti-symmetry of $\mathcal{B}$, we write
\begin{align*}
\braket{\nabla\leps{\theta}\cdot\vareps{\rho}\leps{B}\nabla V_l}_{I_\delta}&=\varepsilon\braket{\vareps{\rho}\partial_i\leps{\theta}\partial_k\vareps{\mathcal{B}}_{l,kij}\partial_jV_l}_{I_\delta}\\
&=\varepsilon\braket{\partial_i\leps\theta\partial_k(\vareps\rho\vareps{\mathcal{B}}_{l,kij})\partial_jV_l}_{I_\delta}-\varepsilon\braket{\partial_i\leps\theta\vareps{\mathcal{B}}_{l,kij}\partial_k\vareps{\rho}\partial_jV_l}_{I_\delta}\\  &=\varepsilon\braket{\partial_i\leps\theta\vareps{\mathcal{B}}_{l,ikj}\partial_k\vareps{\rho}\partial_jV_l}_{I_\delta}-\varepsilon\braket{\rho^\varepsilon\partial_{ik}\leps\theta\vareps{\mathcal{B}}_{l,kij}\partial_jV_l}_{I_\delta}\\   &=\varepsilon\braket{\nabla\leps\theta\cdot\leps{\mathcal{B}}(\nabla\vareps{\rho}\otimes\nabla V_l)}_{I_\delta}.
\end{align*}
Therefore, the second line of~\cref{eq:thtlepsH210} is bounded by
\begin{align*}
    (\leps{B}\nabla V_l,\nabla\leps\theta)_{I_\delta}&=((1-\vareps\rho)\leps{B}\nabla V_l,\nabla\leps\theta)_{I_\delta}+\varepsilon(\leps{\mathcal{B}}(\nabla\vareps{\rho}\otimes\nabla V_l),\nabla\leps\theta)_{I_\delta}\\
    &\le C\abs{\nabla V_l}\nm{\nabla\leps\theta}{L^2(I_\delta)}(\nm{\leps{B}}{L^2(I_\delta^\varepsilon)}+\nm{\leps{\mathcal{B}}}{L^2(I_\delta^\varepsilon)})\\
    &\le\dfrac14\nm{\leps\theta}{\iota}^2+C\dfrac\varepsilon\delta\nm{\nabla V_l}{L^2(I_\delta)}^2.
\end{align*}

Next, we turn to the third line of~\cref{eq:thtlepsH210}. Since $\leps{V}-V_l$ is periodic, it is straightforward to verify 
\[
    (\leps A\nabla\leps V,\nabla(\leps{V}-V_l))_{I_{\kappa\varepsilon}}+\iota^2(\nabla^2\leps V,\nabla^2\leps V)_{I_{\kappa\varepsilon}}=0
\]
for the second term. Similarly, we use~\cref{eq:VlepsAvrg} for the first term. Using~\cref{eq:VlepsH3}, we get
\begin{align*}
    (\bar{A}^\gamma(\boldsymbol{x}_l)\nabla V_l,&(\nabla\leps{V}-V_l))_{I_\delta}-\leps{a}(\leps{V},\leps{V}-V_l)\\
    &\le C\lr{\nm{\nabla V_l}{L^2(I_\delta^\varepsilon)}^2+\nm{\nabla\leps{V}}{L^2(I_\delta^\varepsilon)}^2+\iota^2\nm{\nabla^2\leps V}{L^2(I_\delta^\varepsilon)}^2}\\
    &\le C\dfrac\varepsilon\delta\nm{\nabla V_l}{L^2(I_\delta)}^2.
\end{align*}

Finally, we estimate the last line in~\cref{eq:thtlepsH210}. Choosing 
\[
z=(\leps{V}-V_l)-(\leps{V}-V_l)(1-\vareps\rho)-(I-I_h)[\vareps\rho(\leps V-V_l)]\in V_h\cap H_0^2(I_\delta)
\]
in~\cref{eq:vleps} and~\cref{eq:Vleps}, we obtain
\begin{align*}
    \leps{a}(\leps{\theta},\leps{V}-V_l)&=\leps{a}(\leps{\theta},(\leps{V}-V_l)(1-\vareps\rho)+(I-I_h)[\vareps\rho(\leps V-V_l)])\\
    &\le C\nm{\leps\theta}{\iota}\lr{\nm{(\leps{V}-V)(1-\vareps\rho)}{\iota}+\nm{(I-I_h)[\vareps\rho(\leps V-V_l)]}{\iota}}\\
    &\le\dfrac14\nm{\leps\theta}{\iota}^2+C\lr{\dfrac\varepsilon\delta+\dfrac{h^2}{\varepsilon^2}}\nm{\nabla V_l}{L^2(I_\delta)}^2,
\end{align*}
where we have used~\cref{eq:VlepsH2} and~\cref{eq:VlepsInter} in the last step.

Substituting all the above inequalities into~\cref{eq:thtlepsH210}, we obtain~\cref{eq:thtlepsH21}.
\end{proof}

\begin{remark}
When $\gamma >1$, the sequence $\vareps A$ $H$-converges to the second-order homogenization limit as $\varepsilon\to 0$. However, as shown in~\cite{Tai:2001,LiMingWang:2021}, boundary layer degenerates the convergence rate of the discretization for general singular perturbations.  This implies that for $\gamma \to \infty$, $\nm{\vareps v - I_h\vareps v}{\iota}$ is only $\mathcal{O}(\sqrt{h/\varepsilon})$, without any smoothness assumption on $\vareps v$; see~\cref{sec:apd}. While the discretization error of the corrector given by~\cref{eq:thtlepsH21} is $\mathcal{O}(h/\varepsilon)$.
\end{remark}

We are ready to estimate $\braket{\leps{A}\nabla\leps{\theta}}_{I_\delta}$.
\begin{lemma}\label{lema:rsn1}
If $\iota=\mu\varepsilon^\gamma$ with $\gamma=1$, or $\gamma>1$ and $\nm{\nabla_{\boldsymbol{y}}A}{L^\infty(\Omega\times Y)}$ is bounded, then
\begin{equation}\label{eq:lemaRsn1}
    \abs{\braket{\leps{A}\nabla\leps{\theta}}_{I_\delta}}\le C\Lr{\dfrac\varepsilon\delta+\dfrac{h^2}{\varepsilon^2}}\abs{\nabla V_l}.
\end{equation}
\end{lemma}

\begin{proof}
For any constant vector $\nabla W_l$,  using the fact that the homogenized effective matrix for the dual problem of~\cref{eq:ueps} is $\bar{A}^{\top}$ as shown in~\cite[Lemma 2.2.5]{Shen:2018}, we write
\begin{equation}\label{eq:lemaRsn10}\begin{aligned}
    \abs{I_\delta}\nabla W_l&\cdot\braket{\leps{A}\nabla\leps\theta}_{I_\delta}=\leps{a}(\leps\theta,W_l)=\leps{a}(\leps\theta,\leps{W})-\leps{a}(\leps\theta,\leps{W}-W_l)\\
    &=\iota^2(\nabla^2\leps\theta,\nabla^2\leps W)_{I_\delta}+(\nabla\leps\theta,(\leps{A})^\top\nabla\leps{W}-(\bar{A}^\gamma)^\top(\boldsymbol{x}_l)\nabla W_l)\\
    &\qquad+(\bar{A}^\gamma(\boldsymbol{x}_l)\nabla\leps\theta,\nabla W_l)_{I_\delta}-\leps{a}(\leps\theta,\leps{W}-W_l).
\end{aligned}
\end{equation}
Proceeding along the same line that leads to the first line and the second line in~\cref{eq:thtlepsH210}, we obtain
\begin{align*}
    \iota^2(\nabla^2\leps\theta,\nabla^2\leps W)_{I_\delta}+&(\nabla\leps\theta,(\leps{A})^\top\nabla\leps{W}-(\bar{A}^\gamma)^\top(\boldsymbol{x}_l)\nabla W_l)\\
    &\le C\sqrt{\varepsilon/\delta}\nm{\leps\theta}{\iota}\nm{\nabla W_l}{L^2(I_\delta)}\\
    &\le C\lr{\dfrac\varepsilon\delta+\dfrac{h^2}{\varepsilon^2}}\nm{\nabla V_l}{L^2(I_\delta)}\nm{\nabla W_l}{L^2(I_\delta)},
\end{align*}
where we have used~\cref{eq:thtlepsH21} in the last step.

By~\cref{eq:zepsAvrg} and~\cref{eq:VlepsAvrg}, using~\cref{eq:VlepsH3}, we bound the second to last term in~\cref{eq:lemaRsn10} by
\begin{align*}
(\bar{A}^\gamma(\boldsymbol{x}_l)\nabla\leps\theta,\nabla W_l)_{I_\delta}&=(\bar{A}^\gamma(\boldsymbol{x}_l)\nabla(\leps{V}-V_l),\nabla W_l)_{I_\delta}\\
    &\le C\nm{\nabla(\leps{V}-V_l)}{L^2(I_\delta^\varepsilon)}\nm{\nabla W_l}{L^2(I_\delta^\varepsilon)}\\
    &\le C\dfrac\varepsilon\delta\nm{\nabla V_l}{L^2(I_\delta)}\nm{\nabla W_l}{L^2(I_\delta)}.
\end{align*}

Choosing 
\[
z=\leps{W}-W_l-(\leps{W}-W_l)(1-\vareps\rho)-(I-I_h)[\vareps\rho(\leps W-W_l)]\in V_h\cap H_0^2(I_\delta)
\]
in~\cref{eq:vleps} and~\cref{eq:Vleps}, we estimate the last term in~\cref{eq:lemaRsn10} as
\begin{align*}
    -\leps{a}(\leps\theta,\leps{W}-W_l)&=-\leps{a}(\leps\theta,(\leps{W}-W_l)(1-\vareps\rho)+(I-I_h)[\vareps\rho(\leps W-W_l)])\\
    &\le C\nm{\leps\theta}{\iota}\lr{\nm{(\leps{W}-W_l)(1-\vareps\rho)}{\iota}+\nm{(I-I_h)[\vareps\rho(\leps W-W_l)]}{\iota}}\\
    &\le C\lr{\dfrac\varepsilon\delta+\dfrac{h^2}{\varepsilon^2}}\nm{\nabla V_l}{L^2(I_\delta)}\nm{\nabla W_l}{L^2(I_\delta)},
\end{align*}
where we have used an analogy of~\cref{eq:VlepsH2} and~\cref{eq:VlepsInter} for $\leps{W}$ and~\cref{eq:thtlepsH21} in the last step. 

Substituting the above inequalities into~\cref{eq:lemaRsn10}, we obtain
\begin{align*}
\abs{\nabla W_l\cdot\braket{\leps{A}\nabla\leps{\theta}}_{I_\delta}}&\le\dfrac{C}{\abs{I_\delta}}\Lr{\dfrac{\varepsilon}{\delta}+\dfrac{h^2}{\varepsilon^2}}\nm{\nabla V_l}{L^2(I_\delta)}\nm{\nabla W_l}{L^2(I_\delta)}\\
&\le C\Lr{\dfrac\varepsilon\delta+\dfrac{h^2}{\varepsilon^2}}\abs{\nabla V_l}\abs{\nabla W_l}.
\end{align*}
This leads to~\cref{eq:lemaRsn1}.
\end{proof}

Summing up all the estimates in this part, we are ready to bound $e(\mathrm{HMM})$.
\begin{theorem}\label{thm:eHMM}
If $\iota=\mu\varepsilon^\gamma$ with $\gamma>0$ and $A\in[C^{0,1}(\Omega;L^\infty(\mathbb{R}^d))]^{d\times d}$, then
\begin{equation}\label{eq:thmEHMM}
    e(\mathrm{HMM})\le C\begin{cases}
       \delta+\varepsilon/\delta+\varepsilon^{2(1-\gamma)} & 0<\gamma<1,\\
    \delta+\varepsilon/\delta+h^2/\varepsilon^2 & \gamma=1,\\
\delta+\varepsilon/\delta+\varepsilon^{2(\gamma-1)}+h^2/\varepsilon^2 & \gamma>1,\nm{\nabla_{\boldsymbol{y}}A}{L^\infty(\Omega\times Y)}<\infty.
\end{cases}
\end{equation}
Moreover, if $A^\varepsilon$ is periodic, i.e., $A(\boldsymbol{x},\boldsymbol{y})$ is independent of $\boldsymbol{x}$, and $\delta$ is an integer multiple of $\varepsilon$, then
\begin{equation}\label{eq:thmehmm2}
e(\mathrm{HMM})\le C\begin{cases}
        \varepsilon^{2(1-\gamma)} & 0<\gamma<1,\\
        \varepsilon/\delta+h^2/\varepsilon^2 & \gamma=1,\\    \varepsilon/\delta+\varepsilon^{2(\gamma-1)}+h^2/\varepsilon^2 & \gamma>1,\nm{\nabla_{\boldsymbol{y}}A}{L^\infty(\Omega\times Y)}<\infty.
\end{cases}
\end{equation}
\end{theorem}

\begin{proof}
Substituting~\cref{eq:homoerr} and~\cref{lema:dom},~\cref{lema:loc},~\cref{lema:rsn0},~\cref{lema:rsn1} into~\cref{eq:eHMM}, we obtain~\cref{eq:thmEHMM}. 

If $\leps A=\vareps A$ and $\delta/\varepsilon$ is a positive integer, then error bounds in~\cref{lema:dom} and~\cref{lema:loc} vanish, the estimate~\cref{eq:thmEHMM} changes to~\cref{eq:thmehmm2}.
\end{proof}

\section{Numerical experiments}~\label{sec:numer}
In this part, we report numerical examples to validate the accuracy of the proposed method. We shall employ the Specht element to solve the cell problems because it is one of the best thin plate triangles with $9$ degrees of freedom that currently available~\cite[citation in p. 345]{Zienk:2005}.~\Cref{thm:eHMM} may be extended to the nonconforming microscale discretization with the aid of the enriching operator~\cite{BrennerSung:2005,LiMingWang:2021}. We refer to~\cite{Liao:2024} for more details.

We define a piecewise space as
\[
    H^m_{\mathcal{T}_h}(I_\delta){:}=\set{z_h\in L^2(I_\delta) | z_h|_K\in H^m(K)\quad\text{for all }K\in\mathcal{T}_h},
\]
which is equipped with the broken norm
\[
    \nm{z_h}{H^m_{\mathcal{T}_h}(I_\delta)}^2{:}=\nm{z_h}{L^2(I_\delta)}^2+\sum_{i=1}^m\nm{\nabla_h^iz_h}{L^2(I_\delta)}^2,
\]
where
\[
    \nm{\nabla_h^iz_h}{L^2(I_\delta)}^2=\sum_{K\in\mathcal{T}_h}\nm{\nabla^iz_h}{L^2(K)}^2.
\]
Let $\mathcal{V}_h^i$ be set of the interior vertices of $\mathcal{T}_h$, and $\mathcal{V}_h^b$ be the set of the vertices on the boundary, and $\mathcal{E}_h^i$ be the interior edges of $\mathcal{T}_h$. The Specht element introduced in~\cite{Specht:1988} is an $H^1$-conforming but $H^2$-nonconforming element, which is defined as
\begin{align*}
    X_h{:}=\Bigl\{z_h\in H^1(I_\delta)\ |\ & z_h|_K\in P_K\text{ for all }K\in\mathcal{T}_h,\\
    &\nabla z_h(\boldsymbol{a})\text{ are continuous for all }\boldsymbol{a}\in\mathcal{V}_h^i\Bigr\},
\end{align*}
where $P_K\supset\mathbb{P}_2$ is the local space on $K$. According to~\cite[\S~3.8]{Shi:1987},
\[
\braket{\jump{\partial_{\boldsymbol{n}}z_h}}_e=0\qquad\text{for all}\quad z_h\in X_h,\quad e\in\mathcal{E}_h^i.
\]
The interpolation operator~\cref{eq:inter} is defined in~\cite[Theorem 3]{LiMingWang:2021}.

The bilinear form $\vareps a$ in~\cref{eq:aeps} is replaced by $\heps{a}:H^2_{\mathcal{T}_h}(I_\delta)\times H^2_{\mathcal{T}_h}(I_\delta)\to\mathbb{R}$ given by
\[
    \heps{a}(v_h,z_h){:}=(\vareps{A}\nabla v_h,\nabla z_h)_{I_\delta}+\iota^2(\nabla^2_hv_h,\nabla^2_hz_h)_{I_\delta}\quad \text{for any }v_h,z_h\in H_{\mathcal{T}_h}^2(I_\delta),
\]
and the four type boundary conditions are defined by
\begin{enumerate}
\item Essential boundary condition: 
\[
V_h=\set{z_h\in X_h\cap H_0^1(I_\delta)|\nabla z_h(\boldsymbol{a})=\boldsymbol{0}\text{ for all }\boldsymbol{a}\in\mathcal{V}_h^b};
\]
\item Natural boundary condition: $V_h=X_h\cap H_0^1(I_\delta)$;
\item Free boundary condition: $V_h=\set{z_h\in X_h\cap L_0^2(I_\delta)|\braket{\nabla z_h}_{I_\delta}=\boldsymbol{0}}$;
\item Periodic boundary condition: $V_h=X_h\cap L_0^2(I_\delta)\cap H_{\mathrm{per}}^1(I_\delta)$.
\end{enumerate}
The weighted norm over $V_h$ is defined by
\[
    \nm{\cdot}{\iota,h}{:}=\nm{\nabla\cdot}{L^2(I_\delta)}+\iota\nm{\nabla_h^2\cdot}{L^2(I_\delta)}.
\]
Hence $\heps{a}$ is bounded and covercive over $V_h$ with respect to $\nm{\cdot}{\iota,h}$. By Lax-Milgram theorem, the above cell problem has a unique solution $\heps{v}$.

We select $V_l=x_1$ and $x_2$ in~\cref{eq:aeps}, and calculate the solution $\vareps v_1$ and $\vareps v_2$, respectively. The effect matrix is computed by~\cref{eq:AH}. Motivated by~\cite[\S~2.4]{Yue:2007}, we use the weighted averaging methods to improve the accuracy:
\[
    A_H(\boldsymbol{x}_l){:}=\begin{pmatrix}
        \braket{\omega\vareps A\nabla\vareps v_1}_{I_\delta} &\braket{\omega\vareps A\nabla\vareps v_2}_{I_\delta}
    \end{pmatrix}\begin{pmatrix}
        \braket{\omega\nabla\vareps v_1}_{I_\delta} & \braket{\omega\nabla\vareps v_2}_{I_\delta}
    \end{pmatrix}^{-1},
\]
where the weighted function
\begin{equation}\label{eq:omg}
    \omega(\boldsymbol{x})=\lr{1+\cos(2\pi(x_1-x_{l,1})/\delta)}\lr{1+\cos(2\pi(x_2-x_{l,2})/\delta)}.
\end{equation}
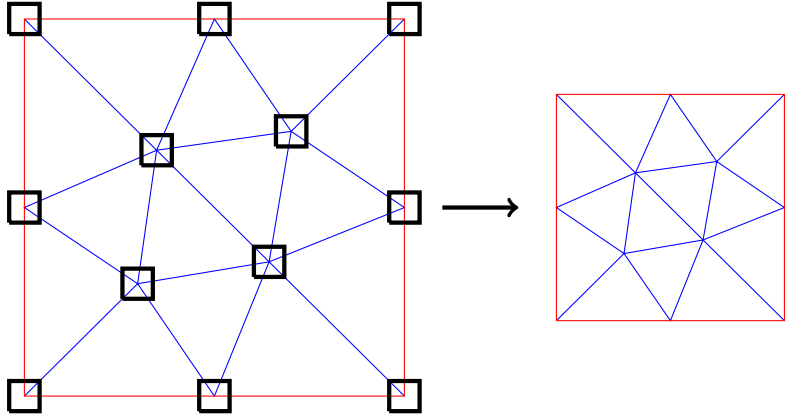
\begin{figure}[htbp]\centering\begin{tikzpicture}
    \draw[blue](0.00,0.00)--(1.49,1.49)--(2.50,0.00)--(3.22,1.78)--(5.00,2.50)--(3.51,3.51)--(2.50,5.00)--(1.74,3.26)--(0.00,2.50)--(1.49,1.49)--(3.22,1.78)--(3.51,3.51)--(1.74,3.26)--(1.49,1.49);
    \draw[blue](5.00,0.00)--(3.22,1.78)--(1.74,3.26)--(0.00,5.00);
    \draw[blue](3.51,3.51)--(5,5);
    \draw[red](0,0)--(5,0)--(5,5)--(0,5)--(0,0);
    \draw[ultra thick](-0.20,-0.20)--(-0.20,0.20)--(0.20,0.20)--(0.20,-0.20)--(-0.20,-0.20);
    \draw[ultra thick](4.80,-0.20)--(4.80,0.20)--(5.20,0.20)--(5.20,-0.20)--(4.80,-0.20);
    \draw[ultra thick](4.80,4.80)--(4.80,5.20)--(5.20,5.20)--(5.20,4.80)--(4.80,4.80);
    \draw[ultra thick](-0.20,4.80)--(-0.20,5.20)--(0.20,5.20)--(0.20,4.80)--(-0.20,4.80);
    \draw[ultra thick](2.30,4.80)--(2.30,5.20)--(2.70,5.20)--(2.70,4.80)--(2.30,4.80);
    \draw[ultra thick](4.80,2.30)--(4.80,2.70)--(5.20,2.70)--(5.20,2.30)--(4.80,2.30);
    \draw[ultra thick](2.30,-0.20)--(2.30,0.20)--(2.70,0.20)--(2.70,-0.20)--(2.30,-0.20);
    \draw[ultra thick](-0.20,2.30)--(-0.20,2.70)--(0.20,2.70)--(0.20,2.30)--(-0.20,2.30);
    \draw[ultra thick](3.02,1.58)--(3.02,1.98)--(3.42,1.98)--(3.42,1.58)--(3.02,1.58);
    \draw[ultra thick](3.31,3.31)--(3.31,3.71)--(3.71,3.71)--(3.71,3.31)--(3.31,3.31);
    \draw[ultra thick](1.54,3.06)--(1.54,3.46)--(1.94,3.46)--(1.94,3.06)--(1.54,3.06);
    \draw[ultra thick](1.29,1.29)--(1.29,1.69)--(1.69,1.69)--(1.69,1.29)--(1.29,1.29);
    \draw[blue](7.00,1.00)--(7.89,1.89)--(8.50,1.00)--(8.93,2.07)--(10.00,2.50)--(9.11,3.11)--(8.50,4.00)--(8.04,2.96)--(7.00,2.50)--(7.89,1.89)--(8.93,2.07)--(9.11,3.11)--(8.04,2.96)--(7.89,1.89);
    \draw[blue](10.00,1.00)--(8.93,2.07)--(8.04,2.96)--(7.00,4.00);
    \draw[blue](9.11,3.11)--(10,4);
    \draw[red](7,1)--(10,1)--(10,4)--(7,4)--(7,1);
    \draw[ultra thick,->](5.5,2.5)--(6.5,2.5);
\end{tikzpicture}\caption{Left: macroscopic meshes $\mathcal{T}_H$ on domain $\Omega$; Right: microscopic meshes $\mathcal{T}_h$ on the cell $I_\delta$.}\label{fig:meshes}\end{figure}

Let $\Omega=(0,1)^2$ and $\mu=1$, we divide $\Omega$ into non-uniform meshes, solve the cell problems on all vertices $\boldsymbol{a}$ in parallel, and $\omega_l=1/3$; See~\cref{fig:meshes}.
\subsection{Accuracy of the effective matrix}
In the first example, we test a layered material, i.e., the coefficient $A^\varepsilon$ depends only on $x_1$,
\[
    A^\varepsilon(\boldsymbol{x})=\dfrac1{2\pi - \cos(2\pi x_1/\varepsilon)}\begin{pmatrix}
        50 + 4\pi^2\cos(2\pi x_1/\varepsilon) & (4\pi^2 + 25/\pi) \sin(2\pi x_1/\varepsilon)\\
        0 & 4\pi^2-1 + \sin(2\pi x_1/\varepsilon)
    \end{pmatrix}.
\]
A direct computation gives the analytical expression of the effective matrix
\[
    \bar{A}=\begin{pmatrix}
        \bar A_{11} & 0\\
        0 & \sqrt{4\pi^2-1}
    \end{pmatrix}
\]
with
\[\bar A_{11}=\begin{cases}
    (50+8\pi^3)/\sqrt{4\pi^2-1}-4\pi^2 & 0<\gamma<1,\\
    25/\pi & \gamma=1,\\
    4\pi^2/[(25+4\pi^3)/\sqrt{625-4\pi^4}-1] & \gamma>1,
\end{cases}
\]

We are interested in whether the resonance error and the discretization error are optimal. Hence, we solely contemplate on the scenario that $\vareps A$ is periodic, and do not pay attention to the error caused by the local periodicity.

We report the relative error
\[
    e_F(\mathrm{HMM}){:}=\max_{\boldsymbol{a}\text{ the vertices of }\mathcal{T}_H}\dfrac{\nm{(A_H-\bar{A})(\boldsymbol{a})}{F}}{\nm{\bar A(\boldsymbol{a})}{F}},
\]
where $\nm{\cdot}{F}$ represents the Frobenius norm of the matrix. Note that $e_F(\mathrm{HMM})$ is equivalent to $e(\mathrm{HMM})$, while it is simpler for implementation.

\begin{table}[htbp]\footnotesize\centering\caption{$e_F$(HMM) w.r.t. the cell size $\delta$ for the 1st example, with and without weighted averaging, $\varepsilon=2^{-5}$ and $h=2^{-9}$.}\label{tab:dlt}\begin{tabular}{l|ccccc}
    \hline $\gamma\backslash\delta$ & $2^{-5}$ & $2^{-4}$ & $2^{-3}$ & $2^{-2}$ & $2^{-1}$\\
    \hline
    &\multicolumn{4}{l}{Essential boundaries without weighted average}\\
    0.25 & 9.8755e-05 & 2.3796e-04 & 3.1778e-04 & 3.5597e-04 & 3.7392e-04\\
    rate &  & -1.27 & -0.42 & -0.16 & -0.07\\
    1.00 & 4.7228e-02 & 2.7798e-02 & 1.4240e-02 & 7.1454e-03 & 3.6163e-03\\
    rate &  & 0.76 & 0.97 & 0.99 & 0.98\\
    4.00 & 2.3679e-01 & 1.1183e-01 & 5.0872e-02 & 2.3541e-02 & 1.1154e-02\\
    rate &  & 1.08 & 1.14 & 1.11 & 1.08\\ \hline
    &\multicolumn{4}{l}{Natural boundaries without weighted average}\\
    0.25 & 3.6687e-01 & 4.1168e-04 & 4.9067e-04 & 5.1721e-04 & 5.0822e-04\\
    rate &  & 9.80 & -0.25 & -0.08 & 0.03\\
    1.00 & 2.7187e-02 & 1.3792e-02 & 6.9795e-03 & 3.7269e-03 & 2.2220e-03\\
    rate &  & 0.98 & 0.98 & 0.91 & 0.75\\
    4.00 & 2.2666e-01 & 1.0711e-01 & 4.8668e-02 & 2.2475e-02 & 1.0629e-02\\
    rate &  & 1.08 & 1.14 & 1.11 & 1.08\\ \hline
    &\multicolumn{4}{l}{Free boundaries without weighted average}\\
    0.25 & 9.7224e-04 & 9.6194e-04 & 9.3580e-04 & 8.5284e-04 & 6.9166e-04\\
    rate &  & 0.02 & 0.04 & 0.13 & 0.30\\
    1.00 & 5.3796e-02 & 3.0771e-02 & 1.5599e-02 & 7.7678e-03 & 3.8279e-03\\
    rate &  & 0.81 & 0.98 & 1.01 & 1.02\\
    4.00 & 9.2755e-03 & 3.1555e-03 & 1.4309e-03 & 6.5085e-04 & 3.0399e-04\\
    rate &  & 1.56 & 1.14 & 1.14 & 1.10\\ \hline
    &\multicolumn{4}{l}{Essential boundaries with weighted average~\cref{eq:omg}}\\
    0.25 & 3.6681e-01 & 3.7678e-04 & 3.8763e-04 & 3.8998e-04 & 3.9050e-04\\
    rate &  & 9.93 & -0.04 & -0.01 & -0.00\\
    1.00 & 4.2970e-01 & 4.1742e-03 & 8.3335e-04 & 2.4794e-04 & 1.7062e-04\\
    rate &  & 6.69 & 2.32 & 1.75 & 0.54\\
    4.00 & 7.7975e-01 & 1.2809e-02 & 1.8316e-03 & 2.9735e-04 & 6.6755e-05\\
    rate &  & 5.93 & 2.81 & 2.62 & 2.16\\ \hline
\end{tabular}\end{table}

Without employing the weighted average, it follows from~\cref{tab:dlt} that $e_F$(HMM) are $\mathcal{O}(1)$ when $\gamma<1$ and $\mathcal{O}(\delta^{-1})$ when $\gamma\ge 1$, which are aligning perfectly with~\cref{thm:eHMM}. 

We observe that the cell problems with free boundary conditions perform slightly better when $\gamma>1$, which seems caused by the boundary layer effects. Moreover, the weighted average may lead to a remarkable reduction in errors for large $\gamma$.

\begin{table}[htbp]\footnotesize\centering\caption{$e_F$(HMM) w.r.t. the multiscale $\varepsilon$ for the 1st example, essential boundary conditions without weighted average, $\delta=2^{-1}$ and $h=2^{-9}$.}\label{tab:epsEssential}
\begin{tabular}{l|ccccc}
    \hline $\gamma\backslash\varepsilon$ & $2^{-1}$ & $2^{-2}$ & $2^{-3}$ & $2^{-4}$ & $2^{-5}$\\ \hline
    0.25 & 6.6896e-03 & 5.5478e-03 & 2.5850e-03 & 1.0149e-03 & 3.7392e-04\\
    rate &  & 0.27 & 1.10 & 1.35 & 1.44\\
    0.50 & 9.2802e-03 & 1.0819e-02 & 7.1842e-03 & 4.0175e-03 & 2.1022e-03\\
    rate &  & -0.22 & 0.59 & 0.84 & 0.93\\
    0.75 & 1.2781e-02 & 2.0616e-02 & 1.9371e-02 & 1.5433e-02 & 1.1503e-02\\
    rate &  & -0.69 & 0.09 & 0.33 & 0.42\\
    1.00 & 4.5901e-02 & 2.7590e-02 & 1.4062e-02 & 7.0174e-03 & 3.6163e-03\\
    rate &  & 0.73 & 0.97 & 1.00 & 0.96\\
    1.50 & 4.9019e-01 & 4.0570e-01 & 2.9837e-01 & 2.0229e-01 & 1.3040e-01\\
    rate &  & 0.27 & 0.44 & 0.56 & 0.63\\
    2.00 & 4.6424e-01 & 2.8744e-01 & 1.3023e-01 & 5.2731e-02 & 2.0802e-02\\
    rate &  & 0.69 & 1.14 & 1.30 & 1.34\\
    4.00 & 3.1211e-01 & 1.0780e-01 & 4.8341e-02 & 2.3041e-02 & 1.1154e-02\\
    rate &  & 1.53 & 1.16 & 1.07 & 1.05\\ \hline
\end{tabular}\end{table}
\begin{table}[htbp]\footnotesize\centering\caption{$e_F$(HMM) w.r.t. the multiscale $\varepsilon$ for the 1st example, natural boundary conditions without weighted average, $\delta=2^{-1}$ and $h=2^{-9}$.}\label{tab:epsNatural}
\begin{tabular}{l|ccccc}
    \hline $\gamma\backslash\varepsilon$ & $2^{-1}$ & $2^{-2}$ & $2^{-3}$ & $2^{-4}$ & $2^{-5}$\\ \hline
    0.25 & 1.4923e-02 & 9.5660e-03 & 4.2174e-03 & 1.5986e-03 & 5.0822e-04\\
    rate &  & 0.64 & 1.18 & 1.40 & 1.65\\
    0.50 & 2.0190e-02 & 1.7877e-02 & 1.0954e-02 & 5.7296e-03 & 2.5736e-03\\
    rate &  & 0.18 & 0.71 & 0.93 & 1.15\\
    0.75 & 2.6924e-02 & 3.1982e-02 & 2.6782e-02 & 1.9528e-02 & 1.2810e-02\\
    rate &  & -0.25 & 0.26 & 0.46 & 0.61\\
    1.00 & 2.7033e-02 & 1.3621e-02 & 6.8198e-03 & 3.6600e-03 & 2.2220e-03\\
    rate &  & 0.99 & 1.00 & 0.90 & 0.72\\
    1.50 & 4.4854e-01 & 3.6039e-01 & 2.6744e-01 & 1.8583e-01 & 1.2302e-01\\
    rate &  & 0.32 & 0.43 & 0.53 & 0.60\\
    2.00 & 4.1122e-01 & 2.4308e-01 & 1.1438e-01 & 4.7150e-02 & 1.8190e-02\\
    rate &  & 0.76 & 1.09 & 1.28 & 1.37\\
    4.00 & 2.7173e-01 & 1.0308e-01 & 4.7668e-02 & 2.2546e-02 & 1.0629e-02\\
    rate &  & 1.40 & 1.11 & 1.08 & 1.08\\ \hline
\end{tabular}\end{table}
\begin{table}[htbp]\footnotesize\centering\caption{$e_F$(HMM) w.r.t. the multiscale $\varepsilon$ for the 1st example, free boundary conditions without weighted average, $\delta=2^{-1}$ and $h=2^{-9}$.}\label{tab:epsFree}
\begin{tabular}{l|ccccc}
    \hline $\gamma\backslash\varepsilon$ & $2^{-1}$ & $2^{-2}$ & $2^{-3}$ & $2^{-4}$ & $2^{-5}$\\ \hline
    0.25 & 5.6414e-02 & 2.0473e-02 & 7.2856e-03 & 2.4598e-03 & 6.9166e-04\\
    rate &  & 1.46 & 1.49 & 1.57 & 1.83\\
    0.50 & 7.2549e-02 & 3.4936e-02 & 1.6430e-02 & 7.3286e-03 & 2.9035e-03\\
    rate &  & 1.05 & 1.09 & 1.16 & 1.34\\
    0.75 & 9.1062e-02 & 5.6303e-02 & 3.5530e-02 & 2.2386e-02 & 1.3497e-02\\
    rate &  & 0.69 & 0.66 & 0.67 & 0.73\\
    1.00 & 5.7404e-02 & 3.1341e-02 & 1.7122e-02 & 9.1322e-03 & 3.8279e-03\\
    rate &  & 0.87 & 0.87 & 0.91 & 1.25\\
    1.50 & 3.9104e-01 & 3.1895e-01 & 2.4111e-01 & 1.7063e-01 & 1.1453e-01\\
    rate &  & 0.29 & 0.40 & 0.50 & 0.58\\
    2.00 & 3.1896e-01 & 1.7056e-01 & 7.2766e-02 & 2.5657e-02 & 7.7513e-03\\
    rate &  & 0.90 & 1.23 & 1.50 & 1.73\\
    4.00 & 7.3885e-02 & 6.9724e-03 & 2.5177e-03 & 9.4015e-04 & 3.0399e-04\\
    rate &  & 3.41 & 1.47 & 1.42 & 1.63\\ \hline
\end{tabular}\end{table}
\begin{table}[htbp]\footnotesize\centering\caption{$e_F$(HMM) w.r.t. the multiscale $\varepsilon$ for the 1st example, periodic boundary conditions without weighted average, $\delta=\varepsilon$ and $h=2^{-8}\varepsilon$.}\label{tab:epsPeriodic}
\begin{tabular}{l|ccccc}
    \hline $\gamma\backslash\varepsilon$ & $2^{-5}$ & $2^{-6}$ & $2^{-7}$ & $2^{-8}$ & $2^{-9}$\\ \hline
    0.25 & 3.9187e-04 & 1.3874e-04 & 4.9131e-05 & 1.8237e-05 & 6.4369e-06\\
    rate &  & 1.50 & 1.50 & 1.43 & 1.50\\
    0.50 & 2.2037e-03 & 1.1058e-03 & 5.5410e-04 & 2.7795e-04 & 1.3983e-04\\
    rate &  & 0.99 & 1.00 & 1.00 & 0.99\\
    0.75 & 1.2068e-02 & 8.6315e-03 & 6.1534e-03 & 4.3773e-03 & 3.1074e-03\\
    rate &  & 0.48 & 0.49 & 0.49 & 0.49\\
    1.00 & 7.9881e-07 & 6.5756e-07 & 6.7266e-07 & 1.1896e-06 & 1.6055e-06\\
    rate &  & 0.28 & -0.03 & -0.82 & -0.43\\
    1.50 & 1.1418e-01 & 7.2736e-02 & 4.4160e-02 & 2.5597e-02 & 1.4225e-02\\
    rate &  & 0.65 & 0.72 & 0.79 & 0.85\\
    2.00 & 7.6310e-03 & 2.0436e-03 & 5.2185e-04 & 1.3121e-04 & 3.2850e-05\\
    rate &  & 1.90 & 1.97 & 1.99 & 2.00\\
    2.50 & 2.6191e-04 & 3.2850e-05 & 4.1084e-06 & 5.1389e-07 & 6.4574e-08\\
    rate &  & 3.00 & 3.00 & 3.00 & 2.99\\ \hline
\end{tabular}\end{table}

~\Cref{thm:eHMM} illustrates how the combined effect of the singular perturbation and the homogenization when $\iota\to 0$ in different regimes. It follows from~\cref{tab:epsEssential},~\cref{tab:epsNatural} and~\cref{tab:epsFree} that the convergence rates of $e_F$ (HMM) are nearly $\mathcal{O}(\varepsilon^{2(1-\gamma)})$ when $\gamma<1$. Due to the influence of the resonance error, it reduces the convergence rate when $\gamma>1$. 

To eliminate the effect of resonance error, we tested the periodic boundary conditions as presented in Table {tab:epsPeriodic}. The table shows that $e_F$(HMM) are $\mathcal{O}(\varepsilon^{2\abs{1-\gamma}})$. When $\gamma=1$, relative errors are small and the rates of convergence are independent of $\varepsilon$. 
\begin{table}[htbp]\footnotesize
    \centering
\caption{$e_F$(HMM) w.r.t. the microscopic mesh size $h$ for the 1st example, with weighted average, $\varepsilon=2^{-5}$ and $\delta=2^{-1}$.}\label{tab:ehHMM}
 \begin{tabular}{l|ccccc}
        \hline $\gamma\backslash h$ & $2^{-5}$ & $2^{-6}$ & $2^{-7}$ & $2^{-8}$ & $2^{-9}$\\ \hline
        &\multicolumn{4}{l}{Essential boundary condition}\\
        0.25 & 1.6364e-04 & 3.0990e-04 & 3.7226e-04 & 3.8698e-04 & 3.9050e-04\\
        rate &  & -0.92 & -0.26 & -0.06 & -0.01\\
        1.00 & 4.6743e-02 & 1.1214e-02 & 2.5144e-03 & 6.1601e-04 & 1.7062e-04\\
        rate &  & 2.06 & 2.16 & 2.03 & 1.85\\
        4.00 & 3.4707e-01 & 7.4217e-02 & 6.0430e-03 & 3.5164e-04 & 6.6755e-05\\
        rate &  & 2.23 & 3.62 & 4.10 & 2.40\\ \hline
        &\multicolumn{4}{l}{Natural boundary condition}\\
        0.25 & 1.6295e-04 & 3.0940e-04 & 3.7225e-04 & 3.8701e-04 & 3.9053e-04\\
        rate &  & -0.93 & -0.27 & -0.06 & -0.01\\
        1.00 & 4.6750e-02 & 1.1218e-02 & 2.5145e-03 & 6.1755e-04 & 1.7228e-04\\
        rate &  & 2.06 & 2.16 & 2.03 & 1.84\\
        4.00 & 3.4660e-01 & 7.4208e-02 & 6.0320e-03 & 3.4621e-04 & 6.4786e-05\\
        rate &  & 2.22 & 3.62 & 4.12 & 2.42\\ \hline
        &\multicolumn{4}{l}{Free boundary condition}\\
        0.25  & 1.6314e-04 & 3.0936e-04 & 3.7236e-04 & 3.8710e-04 & 3.9060e-04\\
        rate &  & -0.92 & -0.27 & -0.06 & -0.01\\
        1.00 & 4.6683e-02 & 1.1185e-02 & 2.5002e-03 & 6.0859e-04 & 1.6431e-04\\
        rate &  & 2.06 & 2.16 & 2.04 & 1.89\\
        4.00 & 3.4655e-01 & 7.4172e-02 & 6.0132e-03 & 3.1711e-04 & 3.1241e-05\\
        rate &  & 2.22 & 3.62 & 4.25 & 3.34\\ \hline
        &\multicolumn{4}{l}{Periodic boundary condition}\\
        0.25 & 1.6411e-04 & 3.1017e-04 & 3.7241e-04 & 3.8710e-04 & 3.9061e-04\\
        rate &  & -0.92 & -0.26 & -0.06 & -0.01\\
        1.00 & 4.6695e-02 & 1.1188e-02 & 2.5004e-03 & 6.0227e-04 & 1.5809e-04\\
        rate &  & 2.06 & 2.16 & 2.05 & 1.93\\
        4.00 & 3.4656e-01 & 7.4172e-02 & 6.0110e-03 & 3.1725e-04 & 2.0052e-05\\
        rate &  & 2.22 & 3.63 & 4.24 & 3.98\\ \hline
    \end{tabular} 
\end{table}

Finally, we employ the weighted average to mitigate the resonance errors, primarily focusing on the error caused by the cell discretization. It followed from~\cref{tab:ehHMM} that $e_F$(HMM) are $\mathcal{O}(1)$ when $\gamma<1$ and at least $\mathcal{O}(h^2)$ when $\gamma\ge 1$, which is consistent with ~\cref{thm:eHMM}. Despite the stronger boundary layer effect, the performance of the essential boundary conditions are comparable with the performance of the natural boundary condition. It seems to achieve nearly $\mathcal{O}(h^4)$ when $\gamma>1$ since the corrector $\boldsymbol{\chi}$ are smoother than we have expected; See Remark{rmk:smooth}.

\subsection{Accuracy of the homogenized solution}
In the second example, we test the accuracy of HMM for singular perturbation homogenization problems with the locally periodic coefficient
\[
    \vareps A(\boldsymbol{x})=\begin{pmatrix}
        \dfrac{20\pi+4\pi^2\cos(2\pi x_1/\varepsilon)}{2\pi-\cos(2\pi x_1/\varepsilon)} & 2+\sin(2\pi x_1)+\cos(2\pi x_2/\varepsilon)\\
        3+\cos(2\pi x_2)+\sin(2\pi x_1/\varepsilon) & \dfrac{22\pi+4\pi^2\sin(2\pi x_2/\varepsilon)}{2\pi-\sin(2\pi x_2/\varepsilon)}
    \end{pmatrix}.
\]
A direct calculation gives the effective matrix 
\[
    \bar{A}(\boldsymbol{x})=\begin{pmatrix}
        \bar A_{11} & 2+\sin(2\pi x_1)\\
        3+\cos(2\pi x_2) & \bar A_{22}
    \end{pmatrix}
\]
with
\[\bar A_{11}=\begin{cases}
    4\pi(2\pi^2 + 5) / \sqrt{4\pi^2 - 1} - 4\pi^2 & 0<\gamma<1,\\
    10 & \gamma=1,\\
    4\pi^2 / [(5 + 2\pi^2) / \sqrt{25 - \pi^2} - 1] & \gamma>1,
\end{cases}\]
and
\[\bar A_{22}=\begin{cases}
    2\pi(4\pi^2+11) / \sqrt{4\pi^2 - 1} - 4\pi^2 & 0<\gamma<1,\\
    11 & \gamma=1,\\
    4\pi^2 / [(11 + 4\pi^2) / \sqrt{121 - 4\pi^2} - 1] & \gamma>1.
\end{cases}.
\]

Let the solution of the homogenization problem
\(
 \bar{u}(\boldsymbol{x})=\sin(\pi x_1)\sin(\pi x_2),
\)
and we compute the source term $f$ by~\cref{eq:u}.

We solve the microscopic cell problems~\cref{eq:aeps} on cells posed over all vertices of $\mathcal{T}_H$ and the macroscopic problem~\cref{eq:uH} by the vertices based integration scheme~\cite[eq. (2.7)]{Du:2010} with $\boldsymbol{x}_l$ the vertices of $K$ and $\omega_l=1/3$ for $l=1,2,3$. We report the relative $H^1$-error $\nm{\nabla(\bar{u}-u_H)}{L^2(\Omega)}/\nm{\nabla\bar{u}}{L^2(\Omega)}$ in~\cref{tab:uHH1}.
\begin{table}[htbp]\footnotesize
    \centering
\caption{Relative $H^1$-errors w.r.t. the macroscopic mesh size $H$ for the 2nd example, without weighted average, $\varepsilon=2^{-5},\delta=2^{-2}$ and $h=2^{-8}$.}
\label{tab:uHH1}
    \begin{tabular}{l|ccccc}
        \hline $\gamma\backslash H$ & $2^{-1}$ & $2^{-2}$ & $2^{-3}$ & $2^{-4}$ & $2^{-5}$\\ \hline
        &\multicolumn{4}{l}{Essential boundary condition}\\
        0.25 & 8.2439e-01 & 4.7977e-01 & 2.4203e-01 & 1.2392e-01 & 6.1370e-02\\
        rate & & 0.78 & 0.99 & 0.97 & 1.01\\
        1.00 & 8.2026e-01 & 4.7772e-01 & 2.4103e-01 & 1.2357e-01 & 6.1434e-02\\
        rate & & 0.78 & 0.99 & 0.96 & 1.01\\
        4.00 & 8.1222e-01 & 4.7459e-01 & 2.3968e-01 & 1.2412e-01 & 6.3624e-02\\
        rate & & 0.78 & 0.99 & 0.95 & 0.96\\ \hline
        &\multicolumn{4}{l}{Natural boundary condition}\\
        0.25 & 8.2447e-01 & 4.7982e-01 & 2.4206e-01 & 1.2393e-01 & 6.1377e-02\\
        rate & & 0.78 & 0.99 & 0.97 & 1.01\\
        1.00 & 8.2326e-01 & 4.7941e-01 & 2.4177e-01 & 1.2381e-01 & 6.1348e-02\\
        rate & & 0.78 & 0.99 & 0.97 & 1.01\\
        4.00 & 8.1382e-01 & 4.7545e-01 & 2.3992e-01 & 1.2390e-01 & 6.2933e-02\\
        rate & & 0.78 & 0.99 & 0.95 & 0.98\\ \hline&\multicolumn{4}{l}{Free boundary condition}\\
        0.25 & 8.2459e-01 & 4.7989e-01 & 2.4209e-01 & 1.2396e-01 & 6.1389e-02\\
        rate & & 0.78 & 0.99 & 0.97 & 1.01\\
        1.00 & 8.2586e-01 & 4.8089e-01 & 2.4256e-01 & 1.2423e-01 & 6.1639e-02\\
        rate & & 0.78 & 0.99 & 0.97 & 1.01\\
        4.00 & 8.2430e-01 & 4.8134e-01 & 2.4248e-01 & 1.2415e-01 & 6.1602e-02\\
        rate & & 0.78 & 0.99 & 0.97 & 1.01\\ \hline&\multicolumn{4}{l}{Periodic boundary condition}\\
        0.25 & 8.2441e-01 & 4.7978e-01 & 2.4204e-01 & 1.2393e-01 & 6.1372e-02\\
        rate & & 0.78 & 0.99 & 0.97 & 1.01\\
        1.00 & 8.2364e-01 & 4.7960e-01 & 2.4191e-01 & 1.2387e-01 & 6.1350e-02\\
        rate & & 0.78 & 0.99 & 0.97 & 1.01\\
        4.00 & 8.2340e-01 & 4.8078e-01 & 2.4222e-01 & 1.2403e-01 & 6.1508e-02\\
        rate & & 0.78 & 0.99 & 0.97 & 1.01\\ \hline
    \end{tabular}
\end{table}

Even though we do not employ the weighted average, it is accurate enough to achieve the precision we desire. Meanwhile the first order rate of convergences of the relative $H^1$-error has been achieved for $\nm{\bar u-u_H}{H^1}$, as predicted in~\cref{eq:uHH1}.

\section{Conclusion}\label{sec:concl}
We introduce a HMM-FEM for solving a singular perturbation homogenization problem. Our method is robust, as it does not rely on the relation between $\iota$ and $\varepsilon$, and explicit form of $A^\varepsilon$, thereby guaranteeing convergence. Unlike the classical second-order elliptic homogenization problem, our analysis reveals that, in certain scenarios, the resonance error and discretization error depend solely on $\varepsilon$. Additionally, we have observed boundary layer effects which, importantly, do not affect the convergence rate, even when essential boundary conditions are imposed on the cell problems.

\appendix
\section{Estimates for the solution of the cell problems}\label{sec:apd}
In this appendix, we discuss the discretization error of the cell problem, which is generally dominated by the interpolation error of the cell solution, which, unfortunately, may not be entirely reasonable. On the one hand, the regularity of the solution to Problem~\cref{eq:veps} depends on the smoothness of the domain. For general boundary value problems on convex polygonal domains, the solution $\vareps v \not\in H^3(I_\delta)$; see~\cite[Theorem 2]{Blum:1980}. Conversely, the cell is usually a cube. On the other hand, what exacerbates the situation is that, even if $\vareps v \in H^3(I_\delta)$ can be guaranteed for essential boundary value problems when $d=2,3$, the interpolation error is only $\mathcal{O}(\sqrt{h/\varepsilon})$ as $\gamma \to \infty$ due to the boundary layer effect. We shall elaborate on this issue.

As $\varepsilon \to 0$, Problem~\cref{eq:veps} with the essential boundary condition tends to
\begin{equation}\label{eq:v0eps}
\left\{\begin{aligned}
-\mathrm{div}(\vareps A\nabla\vareps v_0)&=0 && \text{in }I_\delta,\\ 
\vareps v_0&=V_l && \text{on }\partial I_\delta.
\end{aligned}\right.
\end{equation}

The following lemma regarding the symmetrical matrix $A$ has been provided in~\cite[Appendix A]{Du:2010scheme}, and we extend it to non-symmetrical matrix $A$.
\begin{lemma}
If $\nm{\nabla_{\boldsymbol{y}}A}{L^\infty(\Omega\times Y)}$is bounded, then
\begin{equation}\label{eq:v0epsH2}
\nm{\nabla\vareps{v_0}}{L^2(I_\delta)}+\varepsilon\nm{\nabla^2\vareps{v_0}}{L^2(I_\delta)}\le C\nm{\nabla  V_l}{L^2(I_\delta)}.
\end{equation}
\end{lemma}

\begin{proof}
Multiplying both sides of~\cref{eq:v0epsH2} by $z=\vareps v_0-V_l$, integration by parts and using~\cref{eq:cvcv}, we obtain
\begin{equation}\label{eq:v0epsH1}
\nm{\nabla\vareps v_0}{L^2(I_\delta)}\le C\nm{\nabla V_l}{L^2(I_\delta)}.
\end{equation}

A direct calculation gives
\[
\mathrm{div}(\vareps A\nabla\vareps v_0)=\mathrm{div}((\vareps A)^\top)\cdot\nabla\vareps v_0+\vareps A:\nabla^2\vareps  v_0=\mathrm{div}((\vareps A)^\top)\cdot\nabla\vareps v_0+\vareps A_S:\nabla^2\vareps v_0,
\]
where the symmetric part is defined as
\(
\vareps A_S{:}=\dfrac12\lr{\vareps A+(\vareps A)^\top}.
\)
Therefore, 
\[
\vareps A_S(\boldsymbol{x})\boldsymbol{\xi}\cdot\boldsymbol{\xi}=\vareps A(\boldsymbol{x})\boldsymbol{\xi}\cdot\boldsymbol{\xi}\ge\lambda\abs{\boldsymbol{\xi}}^2.
\]
We rewrite~\cref{eq:v0eps} as
\[
\left\{\begin{aligned}
-\vareps A_S:\nabla^2(\vareps v_0-V_l)&=\mathrm{div}((\vareps A)^\top)\cdot\nabla\vareps v_0 &&\text{in\;}I_\delta,\\
    \vareps v_0-V_l&=0&&\text{on\;}\partial I_\delta.
\end{aligned}\right.
\]
According to ~\cite[Appendix A]{Du:2010scheme} for the symmetric matrix $\vareps A_S$,
\begin{align*}
\nm{\nabla^2\vareps v_0}{L^2(I_\delta)}&\le C\lr{\nm{\nabla\vareps A_S}{L^\infty(I_\delta)}\nm{\nabla(\vareps  v_0-V_l)}{L^2(I_\delta)}+\nm{\mathrm{div}((\vareps A)^\top)\cdot\nabla\vareps v_0}{L^2(I_\delta)}}\\
&\le C\varepsilon^{-1}\nm{\nabla V_l}{L^2(I_\delta)},
\end{align*}
where we have used~\cref{eq:v0epsH1} in the last step. This gives~\cref{eq:v0epsH2}.
\end{proof}

Next we estimate $\vareps v-\vareps v_0$.
\begin{lemma}
Let $\iota = \mu\varepsilon^\gamma$ with $\gamma > 1$ and $\nm{\nabla_{\boldsymbol {y}}A}{L^\infty(\Omega\times Y)}$ is bounded, and $\vareps v$ is the solution of ~\cref{eq:veps} with the essential boundary conditions for $d=2,3$, then
\begin{equation}\label{eq:vepsv0epsH1}
\nm{\nabla(\vareps v-\vareps v_0)}{L^2(I_\delta)}\le C\varepsilon^{(\gamma-1)/2}\nm{\nabla V_l}{L^2(I_\delta)},
\end{equation}
and
\begin{equation}\label{eq:vepsH2}
\nm{\nabla\vareps v}{\iota}\le
C\varepsilon^{-(1+\gamma)/2}\nm{\nabla V_l}{L^2(I_\delta)}.
\end{equation}
\end{lemma}

\begin{proof}
Using~\cref{eq:v0eps}, we rewrite~\cref{eq:veps} with the essential boundary conditions as
\[\left\{\begin{aligned}
    \Delta^2(\vareps v-V_l)&=\iota^{-2}\mathrm{div}(\vareps A\nabla(\vareps v-\vareps v_0))&&\text{in\;}I_\delta,\\
    \vareps v-V_l&=\partial_{\boldsymbol{n}}(\vareps v-V_l)=0&&\text{on\;}\partial I_\delta.
\end{aligned}\right.\]
Scaling by $\hat{\varepsilon}{:}=\varepsilon/\delta,\hat{\iota}{:}=\iota/\delta,\hat{\boldsymbol{x}}{:}=\boldsymbol{x}/\delta$ and with the notation $\hat{v}^\varepsilon(\hat{\boldsymbol{x}}){:}=\vareps v(\boldsymbol{x}),\hat{V}_l(\hat{\boldsymbol{x}}){:}=V_l(\boldsymbol{x}),\vareps{\hat{v}_0}(\hat{\boldsymbol{x}}){:}=\vareps v_0(\boldsymbol{x})$, we rewrite the above equation as the boundary value problems
\[\left\{\begin{aligned}
    \Delta^2(\hat{v}^\varepsilon-\hat{V}_l)&=\hat{\iota}^{-2}\mathrm{div}(\vareps A(\delta\cdot)\nabla(\hat{v}^\varepsilon-\vareps{\hat{v}_0}))&&\text{in\;} I_1,\\
    \vareps{\hat v}-\hat{V}_l&=\partial_{\boldsymbol n}(\vareps{\hat{v}}-\hat{V}_l)=0&&\text{on\;}\partial I_1.
\end{aligned}\right.\]
By~\cite[Theorem 4.3.10]{Mazya:2010}, we get
\[
\nm{\hat{v}^\varepsilon-\hat{V}_l}{H^3(I_1)}\le C\hat{\iota}^{-2}\nm{\mathrm{div}(\vareps A(\delta\cdot)\nabla(\hat{v}^\varepsilon-\vareps{\hat{v}_0}))}{H^{-1}(I_1)}\le C\hat{\iota}^{-2}\nm{\nabla(\hat{v}^\varepsilon-\vareps{\hat{v}_0})}{L^2(I_1)}.
\]
Rescaling back to $T_\delta$, we obtain
\begin{equation}\label{eq:vepsH3}\begin{aligned}
    \nm{\nabla^3\vareps v}{L^2(I_\delta)}&\le C\delta^{d/2-3}\nm{\nabla^3\hat{v}^\varepsilon}{L^2(I_1)}\le C\delta^{d/2-1}\iota^{-2}\nm{\nabla(\hat{v}^\varepsilon-\vareps{\hat{v}_0})}{L^2(I_1)}\\
    &\le C\iota^{-2}\nm{\nabla(\vareps v-\vareps v_0)}{L^2(I_\delta)}.
\end{aligned}\end{equation}

Proceeding along the same line of~\cite[Theorem 5.2]{Tai:2001}, using the definitions of~\cref{eq:veps} and~\cref{eq:v0eps}, for any $z\in H_0^1(I_\delta)\cap H^2(I_\delta)$, we write 
\[
\iota^2(\nabla^2\vareps v,\nabla^2z)_{I_\delta}+(\vareps A\nabla(\vareps v-\vareps v_0),\nabla z)_{I_\delta}=\iota^2\int_{\partial I_\delta}\partial_{\boldsymbol n}^2\vareps v\partial_{\boldsymbol n}z\mathrm{d}\sigma(\boldsymbol x).
\]
Choosing $z=\vareps v-\vareps v_0$ in the above identity, we obtain
\[
\iota^2\nm{\nabla^2\vareps v}{L^2(I_\delta)}^2+\nm{\nabla(\vareps v-\vareps v_0)}{L^2(I_\delta)}^2\le\iota^2(\nabla^2\vareps v,\nabla^2\vareps v_0)_{I_\delta}-\iota^2\int_{\partial I_\delta}\partial_{\boldsymbol n}^2\vareps v\partial_{\boldsymbol n}\vareps v_0\mathrm{d}\sigma(\boldsymbol x).
\]

The first term may be bounded by~\cref{eq:v0epsH2}, it remains to estimate the boundary term. According to the trace inequality and~\cref{eq:vepsH3},
\begin{align*}
\iota^{3/2}\nm{\partial_{\boldsymbol n}^2\vareps v}{L^2(\partial I_\delta)}&\le  C\iota^{3/2}\lr{\delta^{-1/2}\nm{\partial_{\boldsymbol n}^2\vareps v}{L^2(I_\delta)}+\nm{\partial_{\boldsymbol  n}^2\vareps v}{L^2(I_\delta)}^{1/2}\nm{\nabla^3\vareps v}{L^2(I_\delta)}^{1/2}}\\
&\le C(\iota\nm{\nabla^2\vareps v}{L^2(I_\delta)}+\nm{\nabla(\vareps v-\vareps v_0)}{L^2(I_\delta)}),
\end{align*}
where we have used the relation $\iota\ll\varepsilon\ll\delta$. 

Invoking the trace inequality again and using~\cref{eq:v0epsH2}, we obtain
\begin{align*}
\iota^{1/2}\nm{\partial_{\boldsymbol n}\vareps v_0}{L^2(\partial I_\delta)}&\le  C\iota^{1/2}\lr{\delta^{-1/2}\nm{\partial_{\boldsymbol n}\vareps v_0}{L^2(I_\delta)}+\nm{\partial_{\boldsymbol n}\vareps  v_0}{L^2(I_\delta)}^{1/2}\nm{\nabla^2\vareps v_0}{L^2(I_\delta)}^{1/2}}\\
&\le C\sqrt{\iota/\varepsilon}\nm{\nabla V_l}{L^2(I_\delta)}\\
&\le C\varepsilon^{(\gamma-1)/2}\nm{\nabla  V_l}{L^2(I_\delta)},
\end{align*}
where we have used the fact $\varepsilon\ll\delta$. Summing up all the above estimates, we get
\[
\iota\nm{\nabla^2\vareps v}{L^2(I_\delta)}+\nm{\nabla(\vareps v-\vareps v_0)}{L^2(I_\delta)}\le  C\varepsilon^{(\gamma-1)/2}\nm{\nabla V_l}{L^2(I_\delta)},
\]
This gives~\cref{eq:vepsv0epsH1} and the $H^2$-estimate of~\cref{eq:vepsH2}. 

Substituting~\cref{eq:vepsv0epsH1} into~\cref{eq:vepsH3} we obtain the $H^3-$estimate of~\cref{eq:vepsH2}.
\end{proof}

Interpolate~\cref{eq:vepsv0epsH1} and~\cref{eq:vepsH2} with~\cref{eq:v0epsH2}, we get
\[
\nm{\vareps v}{H^{3/2}(I_\delta)}+\iota\nm{\vareps v}{H^{5/2}(I_\delta)}\le C\varepsilon^{-1/2}\nm{\nabla  V_l}{L^2(I_\delta)}.
\]
By~\cite[Theorem 3]{LiMingWang:2021}, there exists a regularized interpolation operator $I_h:H^1(I_\delta)\to X_h$ such that for any $v\in H^s(I_\delta),1\le s\le 3$,
\[
\nm{\nabla^j (I-I_h)v}{L^2(I_\delta)}\le Ch^{s-j}\nm{\nabla^sv}{L^2(I_\delta)},\qquad 0\le j\le s,
\]
where $j$ is a non-negative integer, which implies
\[
\nm{\vareps v-I_h\vareps v}{\iota,h}\le Ch^{1/2}(\nm{\vareps v}{H^{3/2}(I_\delta)}+\iota\nm{\vareps  v}{H^{5/2}(I_\delta)})\le C\sqrt{h/\varepsilon}\nm{\nabla V_l}{L^2(I_\delta)}.
\]
\bibliography{main}
\end{document}